\documentclass[aos,preprint]{imsart}

\RequirePackage[OT1]{fontenc}
\RequirePackage{amsthm,amsmath}

\usepackage{graphicx,epsfig}

\startlocaldefs
\numberwithin{equation}{section}
\theoremstyle{plain}
\newtheorem{theorem}{Theorem}[section]
\endlocaldefs

\newtheorem{prop}[theorem]{Proposition}
\newtheorem{corollary}[theorem]{Corollary}

\newtheorem{lemma}{Lemma}[section]

\theoremstyle{definition}
\newtheorem{definition}[theorem]{Definition}
\newtheorem{remark}[theorem]{Remark}
\newtheorem{example}[theorem]{Example}

\usepackage{amssymb,amscd}
\usepackage{tikz}
\usepackage{xcolor}
\usepackage{tikz-3dplot}

\makeatletter
    \let\@internalcite\cite
    \def\cite{\def\citeauthoryear##1##2{##1, ##2}\@internalcite}
    \def\shortcite{\def\citeauthoryear##1{##2}\@internalcite}
    \def\@biblabel#1{\def\citeauthoryear##1##2{##1, ##2}[#1]\hfill}
\makeatother

\newcommand{\R}{\mathbb{R}}

\def\calG{\mathcal{G}}
\def\calO{\mathcal{O}}

\def\calU{\mathcal{U}}

\def\ud{{\mathcal U}_d}

\def\Z{\mathbb Z}
\def\C{\mathbb C}

\DeclareMathOperator{\argmin}{argmin}

 \newcommand{\un}{\calO_D/\Sigma_d}

 \newcommand{\ve}{\vec w}
 \newcommand{\vu}{\vec u}
  \newcommand{\vv}{\vec v}

    \newcommand{\vel}{\vec \ell}
    \newcommand{\vx}{\vec x}

\begin{document}

\begin{frontmatter}
\title{Averages of Unlabeled Networks: Geometric Characterization and Asymptotic Behavior}  
\runtitle{Averages of Unlabeled Networks}

\begin{aug}
\author{\fnms{Eric D.} \snm{Kolaczyk}\thanksref{m1}\ead[label=e1]{kolaczyk@bu.edu}},
\author{\fnms{Lizhen} \snm{Lin}\thanksref{m2}\ead[label=e2]{lizhen.lin@nd.edu}},
\author{\fnms{Steven} \snm{Rosenberg}\thanksref{m1}\ead[label=e3]{sr@bu.edu}},
\author{\fnms{Jackson} \snm{Walters}\thanksref{m1}\ead[label=e4]{jackwalt@bu.edu}}
\and
\author{\fnms{Jie} \snm{Xu}\thanksref{m1}\ead[label=e5]{xujie@bu.edu}}

\runauthor{Kolaczyk et al.}

\affiliation{Boston University\thanksmark{m1} and the University of Notre Dame\thanksmark{m2}}

\address{Department of Mathematics and Statistics\\
Boston University, USA\\
\printead{e1}\\
\printead{e3}\\
\printead{e4}\\
\printead{e5}}

\address{Department of Applied and Computational\\ Mathematics and Statistics\\
The University of Notre Dame, USA\\
\printead{e2}}
\end{aug}

\begin{abstract}
It is becoming increasingly common to see large collections of network data objects -- that is, data sets in which a network is viewed as a fundamental unit of observation.  As a result, there is a pressing need to develop network-based analogues of even many of the most basic tools already standard for scalar and vector data.  In this paper, our focus is on averages of unlabeled, undirected networks with edge weights.  Specifically, we (i) characterize a certain notion of the space of all such networks, (ii) describe key topological and geometric properties of this space relevant to doing probability and statistics thereupon, and (iii) use these properties to establish the asymptotic behavior of a generalized notion of an empirical mean under sampling from a distribution supported on this space.  Our results rely on a combination of tools from geometry, probability theory, and statistical shape analysis.  In particular, the lack of vertex labeling necessitates working with a quotient space modding out permutations of labels.  This results in a nontrivial geometry for the space of unlabeled networks, which in turn is found to have important implications on the types of probabilistic and statistical results that may be obtained and the techniques needed to obtain them. 
\end{abstract}

\begin{keyword}[class=MSC]
\kwd[Primary ]{62E20}
\kwd{62G20}
\kwd[; secondary ]{53C20}
\end{keyword}

\begin{keyword}
\kwd{fundamental domain}
\kwd{Fr{\'e}chet mean}
\kwd{graphs}
\end{keyword}

\end{frontmatter}

\section{Introduction}

Over the past 15\textendash 20 years, as the field of network science has exploded with activity, the majority of attention has been focused upon the analysis of (usually large) {\em individual} networks.  See~\cite{jackson2008social,kolaczyk2009statistical,newman2010networks}, for example.  While it is unlikely that the analysis of individual networks will become any less important in the near future, it is likely that in the context of the modern era of `big data' there will soon be an equal need for the analysis of (possibly large) {\em collections} of (sub)networks, i.e., collections of network data objects.

We are already seeing evidence of this emerging trend.  For example, the analysis of massive online social networks like Facebook can be facilitated by local analyses, such as through extraction of ego-networks (e.g., \cite{gjoka2010walking}).  Similarly, the $1000$ Functional Connectomes Project, launched a few years ago in imitation of the data-sharing model long-common in computational biology, makes available a large number of fMRI functional connectivity networks for use and study in the context of computational neuroscience (e.g, \cite{biswal2010toward}).  It would seem, therefore, that in the near future networks of small to moderate size will themselves become standard, high-level data objects.  

Faced with databases in which networks are the fundamental unit of data, it will be necessary to have in place a network-based analogue of the `Statistics 101' tool box, extending standard tools for scalar and vector data to network data objects.  The extension of such classical tools to network-based datasets, however, is not immediate, since networks are not inherently Euclidean objects. Rather, formally they are combinatorial objects, defined simply through two sets, of vertices and edges, respectively, possibly with an additional corresponding set of weights.  Nevertheless, initial work in this area demonstrates that networks can be associated with certain natural Euclidean subsets and furthermore demonstrates that through a combination of tools from geometry, probability theory, and statistical shape analysis it should be possible to develop a comprehensive, mathematically rigorous, and computationally feasible framework for producing the desired analogues of classical tools.

For example, in our recent work~\cite{ginestet2014hypothesis} we have characterized the geometry of the space of all labeled, undirected networks with edge weights, i.e., consisting of graphs $G=(V,E,W)$, for weights $w_{ij}=w_{ji}\ge 0$, where equality with zero holds if and only if $\{i,j\}\notin E$.  This characterization allowed us in turn to establish a central limit theorem for an appropriate notion of a network empirical mean, as well as analogues of classical one- and two-sample hypothesis testing procedures.  Other results of this type include additional work on asymptotics for network empirical means~\cite{tang2016law} and regression modeling with a network response variable, where for the latter there have been both frequentist~\cite{cornea2016regression} and Bayesian~\cite{david-bio} proposals.  Work in this area continues at a quick pace -- see, for example, \cite{levina}, which proposes a classification model based on network-valued inputs and \cite{ddurante2014nonparametric}, which proposes a nonparametric Bayes model for  distributions on populations of networks.  Earlier efforts in this space have focused on the specific case of trees.  Contributions of this nature include work on central limit theorems in the space of phylogenetic trees\\\cite{Billera2011,barden13} and work by Marron and colleagues
\cite{wang2007object,aydin2009principal} in the context of so-called object-oriented data analysis with trees.

To the best of our knowledge, nearly all such work to date pertains to the case of \emph{labeled} networks:  that is, to networks in which the vertices $V$ have unique labels, e.g., $V=\{1,\ldots, d\}$.  In fact, \emph{unlabeled} networks have received decidedly less attention in the network science literature as a whole but nevertheless arise in various important settings.  For example, surveys of so-called `personal networks' are common in social network analysis, wherein individuals (`egos') are surveyed for a list of $d$ others (`alters') with whom they share a certain relationship (e.g., friendship, colleague, etc.) and only common patterns across networks in the structure of the relationships among these others within each network are of interest.  This leads to analyses that either ignore vertex labels or for which vertex labels are simply not available (e.g., through de-identification).  See~\cite{mccarty2002structure}, for example.  Other variations on this idea include, for example, the study of neighborhood subgraphs in the context of online social networks (e.g., \cite{ugander2013subgraph, mcauley2014discovering}), \\
where now the number of vertices $d$ in these neighborhoods typically varies.  More generally, the study of ego-networks (whether inclusive or exclusive of the ego vertex and its edges -- confusingly, the term is used for both cases) under privacy constraints can increasingly be expected to be a natural source of collections of unlabeled networks.

In this paper, our focus is on averages of unlabeled, undirected networks with edge weights.  Adopting a perspective similar to that in our previous work~\cite{ginestet2014hypothesis}, we (i) characterize a certain notion of the space of all such networks, (ii) describe key topological and geometric properties of this space relevant to doing probability and statistics thereupon, and (iii) use these properties to establish the asymptotic behavior of a generalized notion of an empirical mean under sampling from a distribution supported on this space.  In particular, adopting the notion of a Fr{\'e}chet mean, we establish a corresponding strong law of large numbers and a central limit theorem. In contrast to~\cite{ginestet2014hypothesis}, where the corresponding space of networks was found to form a smooth manifold, here the lack of vertex labeling necessitates working with a quotient space modding out permutations of labels.  As a result, we have only an orbifold -- a more general geometric structure -- which in turn is found to have important implications on the types of probabilistic and statistical results that may be obtained and the techniques needed to obtain them.  

The nature of our work is in the spirit of statistics on manifolds and statistical shape analysis, which employs the geometry of manifolds or shape spaces for defining Fr\'echet means and developing large sample theory of their sample counterparts for inference.  See\\
 \cite{rabibook} for a rather comprehensive treatment on the subject.
Our approach to studying the entire  family of networks subject to an equivalence relation under a group action, via forming the associated quotient or moduli space, is a common theme in modern geometry, including gauge theory~\cite{donaldson1990geometry}, symplectic topology~\cite{mcduff2004j}, and algebraic geometry\\\cite{cox1999mirror,vakil2003moduli}.  The appearance of orbifolds, often much more complicated than in our case, is quite common.  

Finally, there is a large literature on graph limits, for which substantial work has been done on analysis of appropriate spaces of networks (e.g., \cite{lovasz2012large}).  While there are high-level connections between graphons as a space of equivalence classes (as studied in the literature) and unlabeled networks as another space of equivalence classes (as studied here), a meaningful comparison between the two is not immediate.  For example, while one natural source of unlabeled networks is as finite-dimensional restrictions of graphons, our framework encompasses both these and more general cases (e.g.,  {\em dependent} generation of (non)edges, conditional on a graphon).  On the other hand, the distance defined in our framework for weighted graphs is the Euclidean norm taken over the equivalence classes of weighted matrices under the group action of the permutation group, while the cut distance commonly used for graphons is defined over the equivalent classes of bivariate functions under the group action of all measure-preserving transformations.  \textcolor{black}{The permutation group is a finite group, as compared to the infinite dimensional  group of all measure-preserving transformations.}    Ultimately, however,  we note too that, whereas the focus in the graphon literature typically is on the case of a single network asymptotically increasing in size (or finite-dimensional restrictions thereof), here the focus is on asymptotics in many networks, with the dimension fixed.  
These and related issues suggest that a simple comparison is unlikely to be forthcoming.

The organization of this paper is as follows.  In Section~\ref{eq-sec2} we present our characterization of the space of unlabeled networks.  Results from our investigation of the asymptotic behavior of the Fr{\'e}chet empirical mean are then provided in Section~\ref{sec:net.aves}.  While a strong law of large numbers is found to emerge under quite general conditions, establishing just when conditions dictated by the current state of the art for central limit theorems on manifolds hold turns out to be a decidedly more subtle exercise.  This latter is the focus of Section~\ref{sec-uniqueness}.  Some additional discussion of open problems may be found in Section~\ref{sec:disc}.  The Appendices discuss implementation issues for the main theoretical results in the paper, and the Supplementary materials contain several results related to the latter.

After we submitted this paper, we became aware of related work of B. Jain.   In~\cite[Thm.~4.2]{Jainb2016} a result is stated that is equivalent to our main result on the uniqueness of the Fr\'echet mean (Theorem 4.5).  However, no proof is given in~\cite{Jainb2016}, although the companion paper~\cite{Jaina2016} contains much of the material needed for a proof.  Here, in Section~\ref{sec-uniqueness}, we give a complete proof of this result.  In addition, we also give two extensions in Supplement C.  Proofs of both the main result and the extensions are nontrivial.  On the computational side, our work gives an algorithm for the computation of fundamental domains (see Appendix A), while
\cite[\S3.3]{Jaina2016} gives an algorithm  for iteratively approximating the sample mean for the Fr\'echet function of a distribution on the space of unlabeled networks.  Both algorithms are  addressing NP-hard problems.

\section{The space of unlabeled networks}
\label{eq-sec2}

Our ultimate focus in this paper is on a certain well-defined notion of an `average' on elements drawn randomly from a `space' of unlabeled networks and on the statistical behavior of such averages.  Accordingly, we need to establish and understand the relevant topology and geometry of this space.  We do so by associating labeled networks with vectors and mapping those to unlabeled networks through the use of equivalence classes in an appropriate quotient space.  In this section we provide relevant definitions, characterization, and illustrations of this space of unlabeled networks.

\subsection{The topological space of unlabeled networks}

 Let $G = (V,E)$ be a labeled, undirected graph/network with weighted edges and with $d$ vertices/nodes.  We always think of $E$ as having $D:= {d\choose 2}$ elements, where some of the edge weights can be zero. We think of the edge weight between vertices $i$ and $j$ as the strength of some unspecified relationship between $i$ and $j$.  
  
 Let $\Sigma_d$ be the group of permutations of $\{1,2,\ldots,d\}.$ A permutation $\sigma \in \Sigma_d$ of the $d$ vertex labels  technically produces a 
  new graph $\sigma G$, but with no new information.  To define $\sigma G$ precisely, 
  note that  the weight function $w_G  :E\to \R_{\geq 0}$ 
  can be thought of as a symmetric function $w_G: V\times V\to \R_{\geq 0}$, with $w_G(i,j)$ 
  the weight of the edge joining vertex $i$ and vertex $j$ in $G$. 
  Therefore
  the action of $\Sigma_d$ on $w_G$ is given by 
  $$(\sigma\cdot w_G)(i,j) = w_G(\sigma^{-1}(i), \sigma^{-1}(j)).$$
  (The inverse guarantees that $(\sigma\tau)\cdot w_G = \sigma\cdot(\tau\cdot 
  w)_G.$)  
  Note that for general $G$, not all permutations of the entries of $w_G$ are of the form $\sigma\cdot w_G,$ as $w_G$ may have $[d(d-1)/2]! $ distinct permutations and $\Sigma_d$ has $d!$ elements

In summary, $\sigma G$ is defined to be the graph on $d$ vertices with weight function 
  $\sigma\cdot w_G:E\to \R_{\geq 0}.$ 
  Let $\calG = \calG_d$ be the set of all labeled graphs with $d$ vertices.  Then the quotient space 
  $$\calU_d = \calG_d/\Sigma_d$$
   is the space of unlabeled graphs, the object we want to study.  This means that an unlabeled network $[G]\in \calU_d$ is
   an equivalence class 
   $$[G] := \{\sigma\cdot G: \sigma\in\Sigma_d\}.$$
   
 As we now explain, $\calG_d$ looks like an explicit subset of $\R^d$, and so is easy to picture.  In contrast, the quotient space $\calU_d$ is difficult to picture.  Nevertheless, as we describe in the following paragraphs, the topology of $\calU_d$ may be characterized through standard point-set topology techniques, with the conclusion that everything works as well as possible. 
 Readers who wish can safely skip to the examples in Section~2.2. 
   
  Fix an ordering of the vertices $1,...,d$, and take the lexicographic ordering 
  $\{(i,j): 1 \leq i < j\leq d\}$ on the set of edges. (Thus $(i,j)<(k,\ell)$ if $i<k$ or $i=k$ and $j<\ell.$)  
 Given this ordering, we get an injection 
$$\alpha:\calG_d\to \R^{D},\ \  \alpha(G) = (w_1(G),\ldots,w_D(G)),$$
where $w_i(G)$ is the weight of the $i^{\rm th}$ edge of $G$.  The image of $\alpha$ is the first
``octant" $\calO_D = \{\vec x = (x^1,\ldots, x^D): x^i \geq 0\}.$ 
For simplicity, we choose the standard Euclidean metric on $\calO_D$.  This pulls back via $\alpha$ to a metric on $\calG_d$ with the desirable property that two networks are close iff their edge weights are close. 
 Similarly, the standard topology on $\calO_D$ (an open ball in 
$\calO_D$ is the intersection of an open $\R^D$-Euclidean ball with $\calO_D$) pulls back to a topology on $\calG_D.$  (This just means that $A\subset \calG_D$ is open iff
$\alpha(U)$ is open in $\calO_D.$  This makes $\alpha$ a homeomorphism.) Just as in $\R^D$, the metric and topology are compatible:  a sequence of graphs/weight vectors $\vec x_i$ in $\calO_D$ converges to a graph/weight vector $\vec x$ in the topology of $\calO_D$ iff the distance from $\vec x_i$ to $\vec x$ goes to zero.

Via the bijection 
$\alpha$, the action of $\Sigma_d$ on $\calG_d$ transfers to an action on $\calO_D.$
First, $\sigma\in \Sigma_d$ acts on $\{1,\ldots, D\}$ by $\sigma\cdot i = j$ if $i$ corresponds to the edge $(i_1, i_2)$ and $j$ corresponds to the edge $(\sigma(i_1), \sigma(i_2))$.  Then $\sigma$ acts on
$\calO_D$ by $\sigma\cdot \vec x = (x^{\sigma^{-1}(1)},\ldots, x^{\sigma^{-1}(D)}).$  
Since we've arranged the actions to be compatible with $\alpha:\calG_d\to\calO_D$, 
we get a well defined bijection $\overline{\alpha}$:
$$\calU_d = \calG_d/\Sigma_d \stackrel{\overline \alpha}{\to} \calO_D/\Sigma_d,\ \overline{\alpha}[G] = [\alpha(G)].$$
From now on, we just denote $\overline{\alpha}$ by $\alpha.$

To complete the topological discussion, we note that $\alpha:\calU_d\to\calO_D/\Sigma_d$ is a homeomorphism if we give both sides the quotient topology: for the map $q: \calG_d\to\calU_d$ taking a graph to its equivalence class, a  set  
  $U\subset \calU_d$ is open iff $q^{-1}(U)$ is open in $\calO_D$.  The quotient topology on $\calO_D/\Sigma_d$ is defined similarly.


\subsection{Examples of quotient spaces}

As a warmup, we first give a simple example of a quotient space resulting from the action of a finite group on a Euclidean space.  This particular example is important in providing a relevant non-network analogy to our network-based results.  We will revisit it frequently throughout the paper.

\begin{example}
\label{ex-r2modz4}
The group $\Z_4 = \{0,1,2,3\}$ acts on the plane $\R^2$ by rotation counterclockwise by $90$ degrees: specifically, for $k\in \Z_4$ and $z\in \R^2 = \C$, 
$$k\cdot z = e^{ik\pi/2}\cdot z.$$
Thus $0\cdot z = z,  1\cdot z = e^{i\pi/2} z,$ etc. A point in the quotient space $\C/\Z_4$ is the 
 set $[z_0] = \{ e^{ik\pi/2}z_0: k\in \Z_4\}.$  The set $[z_0]$ is called the orbit of $z_0$ under $\Z_4.$  Note that every orbit is a four element set except for the exceptional orbit $[\vec 0] = \{\vec 0\}.$

The closed first quadrant $F = \{(x,y): x\geq 0, y\geq 0\}$ is a {\em fundamental domain} for this action; {\em i.e.}, each orbit $[z_0]$ has a unique representative/element in $F$, except possibly for the orbits of points on the boundary $\partial F = \{(x,y): x=0 \ {\rm or}\ y=0\}$ of $F$.  Orbits of boundary points like $[(5,0)]$ have two representatives $(5,0), (0,5)$ in $F$, while the origin of course has only one representative.  
\end{example}

Here is a precise definition of a fundamental domain for the action of a group $G$ on a set $S$:
\begin{definition} \label{2.2}
$F\subset S$ is a fundamental domain for the action of $G$ if 
(i) $S $ is the union of the orbits of $F$ ($S = \cup_{k\in G} k\cdot F$);

\noindent (ii)  orbits can intersect only at boundary points ($k_1\cdot F\cap k_2\cdot F = \emptyset$ or $k_1\cdot F\cap k_2\cdot F
\subset \partial (k_1\cdot F)\cap\partial(k_2\cdot F)$). 
\end{definition}

In this example, $G=\Z_4$ and $S = \C.$ It follows that the quotient map $q:F\to \C/\Z_4$ is surjective, a homeomorphism on the interior of $F$ (where $\C/\Z_4$ has the quotient topology), and finite-to-one on the boundary of $F$.

If we want to picture a set that is bijective to $\C/\Z_4$, we could take e.g. $F'$ to be $F$ minus the positive $y$-axis.  This is not so helpful topologically or geometrically, as the points $[(5,0)]$ and $[(0,5.01)]$ have close representatives $(5,0), (5.01,0)$ in $F$, while their representatives $(5,0)$ and
$(0,5.01)$ are not close in $F'$.  In particular, the sequence $(10^{-k}, 5)$ does not converge in $F'$, but the orbits $[10^{-k}, 5]$ converge to $[0,5] = [5,0]$ in $\C/\Z_4.$  Thus $F'$ does not give us a good picture of $\C/\Z_4$ topologically.

In summary, it is much better to keep both positive axes in $F$, and to consider $\C/\Z_4$ as (in bijection with) $F$ with the boundary points $(a,0)$ and $(0,a)$ ``glued together."  More precisely, we have a bijection
$$\beta:\mathcal F :=  \frac{F}{(a,0)\sim (0,a)} \to \C/\Z_4,$$
where the denominator indicates that the two point set $\{(a,0), (0,a)\}$ ($a\neq 0$) is one point of $\mathcal F$, 
while all other points of $F$ correspond to a single point in $\mathcal F.$  At the price of this gluing, 
we now have that $\beta$ is a homeomorphism: in particular,  $\lim_{i\to\infty} x_i = x$ in $\mathcal F$ iff
$\lim_{i\to\infty} \beta(x_i) = \beta(x)$ in $\C/\Z_4.$
 (Technical remark: $\mathcal F$ gets the quotient topology from the standard topology on $F$ and the obvious surjection $q:F\to\mathcal F.$)

Although this seems a little involved, it is quite easy to perform the gluing in $\mathcal F$ in rubber sheet topology: stretching the interior of $F$ to allow the gluing of the two axes shows that $\mathcal F$ and hence $\C/\Z_4$ is a hollow cone.  See Figure~\ref{fig:CmodZ4}.\footnote{ Color versions of all figures are in Supplement E. In figures below, regions called red appear as white in black-and-white reproductions, and regions called blue appear as gray.}

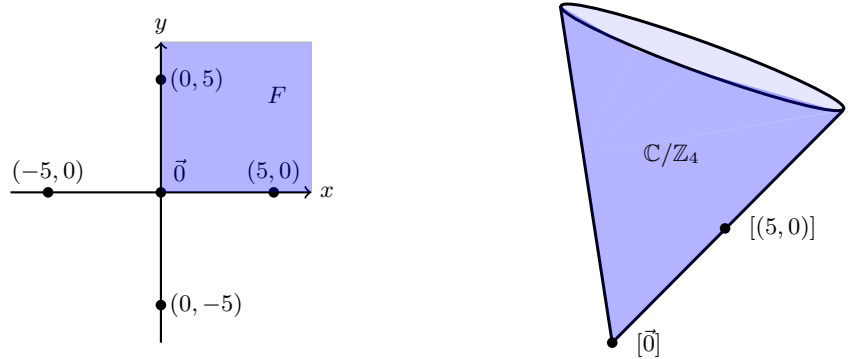
\begin{figure}
$$
\begin{tikzpicture}
\draw[thick, ->](-2,0) -- (2,0);
\node[right] at (2,0) {$x$};
\draw[thick, ->] (0,-2) -- (0,2);
\node[above] at (0,2) {$y$};
\fill (0,0) circle (2pt);
\node[right] at (.05, .3) {$\vec 0$};
\fill (1.5,0) circle (2pt);
\fill (-1.5,0) circle (2pt);
\fill (0,1.5) circle (2pt);
\fill (0,-1.5) circle (2pt);
\node[above] at (1.5,0) {$(5,0)$};
\node[right] at (0,1.5) {$(0,5)$};
\node[above] at (-1.5,0) {$(-5,0)$};
\node[right] at (0,-1.5) {$(0,-5)$};
\node[right] at (1.3,1.3) {$F$};
\filldraw[ fill = blue, opacity = 0.3] (0,0) -- (2,0) -- (2,2) -- (0,2) -- cycle;


\draw[line width=.04 cm,   rotate = -20, shift = {(1.15,5.15)}] (5,-1) ellipse (2.0cm and .2cm);
\filldraw[fill=blue, opacity = 0.1,  rotate = -20, shift = {(1.15,5.15)}] (5,-1) ellipse (2.0cm and .2cm);
\draw[line width =.04 cm] (5.31,2.48) -- (6,-2) -- (9.07, 1.09);
\fill[fill=blue, opacity = 0.3, thick] (5.31,2.48) -- (6,-2) -- (9.07, 1.09) -- (5.6,0.5) -- (6.3,1.9);
\fill[fill=blue, opacity = 0.3] (6.3,1.9) -- (5.6,0.5) -- (7.0, 1.65);
\fill[fill=blue, opacity = 0.3] (7.0, 1.65) -- (5.6,0.5) -- (9.0, 1.08);

\node[right] at (6.2,-2) {$[\vec 0]$};
\node[right] at (7.7,-.48) {$[(5,0)]$};
\fill (6,-2) circle (2pt);
\fill (7.5, -.48) circle (2pt);
\node[right] at (6.3,.5) {$\C/\Z_4$};

\end{tikzpicture}
$$
\caption{ In the figure on the left, $F$ is a fundamental domain $F$ for the action of $\Z_4$ on $\C$.  The four point orbit of $(5,0)$ and the one point orbit of $\vec 0$ are shown.  In the figure on the right, the quotient space $\C/\Z_4$ is drawn as a hollow cone given by taking $F$ and gluing $(x,0)$ to $(0,x)$. }
\label{fig:CmodZ4}
\end{figure}


\begin{example}

We discuss the case of a network with three vertices.   
This is a deceptively easy case, as $3=d = D$ implies that every permutation of the $D$ edge weights comes from a permutation in $\Sigma_d.$   In higher dimensions, the details are more complicated.

In Supplement A, we describe the quotient space $\calU_3$ of unlabeled graphs directly.  However,  
it is easier to picture $\calU_3$  by finding a  fundamental domain  $F$ inside $\calO_3$ for the action of $\Sigma_3.$  As in the previous example,  $F$ is a closed set 
such that the quotient map $q|_F:F\to \calU_3$ is a continuous surjection, a homeomorphism  from
 the interior of $F$ to  its image, and a finite-to-one map on the boundary 
$\partial F$ of $F$. 
Thus $F$ represents $\calU_3$ bijectively except for some gluings on the boundary.  This is illustrated in Figure 2, where 
$F = \{(x,y,z): x\geq y\geq z\geq 0\}.$  Again, the case $d=3$ is deceptively easy, as $F$ is a bijection even on  $\partial F.$
\begin{figure}
$$\tdplotsetmaincoords{60}{140}
\begin{tikzpicture}[tdplot_main_coords,scale = .9]
\node[left] at (8,0,0) {$x$};
\draw[->, ultra thick] (1.2,0,0) -- (8,0,0);
\draw [->, ultra thick] (0,0,0) -- (0,5,0); 
\node[right] at (0,5,0) {$y$};
\draw[-> ,ultra thick ] (0,0,0) -- (0,0,5);
\node[above] at (0,0,5) {$z$};

\filldraw[draw=black, fill = blue, opacity = .9] (0,0,0) -- (7,7,0) -- (9,9,10) --cycle; 
\node[right, blue] at (0,.5,-1) {$x=y$};
\filldraw[draw=black, fill = green, opacity = 0.3] (0,0,0) -- (9,9,10) -- (7,0,0) --cycle;
\node[right,thick, green] at (0,-6,-2) {$y=z$};

\draw[thick,magenta] (0,0,0) -- (9,9,10);\draw[thick, magenta] (7,0,0) -- (9,9,10);
\draw[thick,magenta] (7,7,0) -- (9,9,10); \draw[thick,magenta] (7,0,0) -- (7,7,0);
\draw[thick,magenta] (0,0,0) -- (7,7,0);
\draw[thick, magenta, dashed] (9,9,10) -- (19,19,190/9);
\draw[thick, magenta, dashed] (7,7,0) -- (12,12,0);
\draw[thick, magenta, dashed] (7,0,0) -- (12,0,0);
\draw[thick, magenta, dashed] (10.2,0,0) -- (15.2,15.2,152/9)--(10.4,10.4,0) --cycle;

\draw[fill=yellow, opacity = .2] (0,0,0) -- (7,0,0) -- (7,7,0) -- cycle;
\node[right,thick, orange] at (7,1,-1) {$x=0$};

\draw (9.7,2,0) -- (9.7,3.3,.57) -- (9.7,2.65,1) -- cycle;
 \fill (9.7,2,0) circle (2pt);
  \fill (9.7,3.3,.57) circle (2pt);  
\fill (9.7,2.65,1) circle (2pt);
\node[right] at (9.7, 3, .9) {$1$}; 
\node[below] at (9.7, 2.65, .2) {$3$}; 
\node[left] at (9.7, 2.3, .6) {$2$}; 
\fill (5,3,1.7) circle(2pt);
\node[above] at (5,3,1.7) {$P$};
\draw[->,thick] (8.2,1.9,0) -- (5.2,3.1,1.7);

\draw[fill=yellow, opacity = .2] (9,6,0) -- (9,7.3,.57) -- (9,6.65,1) -- cycle;
 \fill (9,6,0) circle (2pt);
  \fill (9,7.3,.57) circle (2pt);  
\fill (9,6.65,1) circle (2pt);
\node[right] at (9, 7, .9) {$0$}; 
\node[below] at (9, 6.65, .2) {$3$}; 
\node[left] at (9, 6.3, .6) {$2$}; 
\fill (3.9,1.72 ,0) circle (2pt);
\draw[->,thick] (8.7,6.4,1) -- (4.2, 1.95, 0);

\draw[fill=blue, opacity = .9] (3,4,0) -- (3,5.3,.57) -- (3,4.65,1) -- cycle;
 \fill (3,4,0) circle (2pt);
  \fill (3,5.3,.57) circle (2pt);  
\fill (3,4.65,1) circle (2pt);
\node[right] at (2.1, 3.9, 0) {$1$}; 
\node[below] at (2.7, 4.3, 0) {$2$}; 
\node[left] at (2.4, 3.55, 0) {$2$}; 
\fill (3,2.8,1) circle(2pt);
\draw[->,thick] (2.4,3.1,0) -- (3, 3, 1);

\draw [fill = green, opacity = .3](5,-2.2,0) -- (5,-.9,.57) -- (5,-1.5,1) -- cycle;
 \fill (5,-2.2,0) circle (2pt);
  \fill (5,-.9,.57) circle (2pt);  
\fill (5,-1.5,1) circle (2pt);
\node[right] at (5, -1.2, .9) {$2$}; 
\node[below] at (5, -1.5, .3) {$3$}; 
\node[left] at (5, -1.85, .55) {$2$}; 
\fill (4.5,.2, 1.05) circle(2pt);
\draw[->,thick] (4.9,-.75,1.05) -- (4.5, 0, 1);

\draw [fill = red, opacity = .6](3,4,5.2) -- (3,5.3,5.77) -- (3,4.65,6.2) -- cycle;
 \fill (3,4,5.2) circle (2pt);
  \fill (3,5.3,5.77)  circle (2pt);  
\fill (3,4.65,6.2) circle (2pt);
\node[right] at (3, 5, 6.1) {$1$}; 
\node[below] at (3, 4.65, 5.4) {$1$}; 
\node[left] at (3, 4.3, 5.8) {$1$}; 
\fill (4,4, 40/9) circle(2pt);
\draw[->,thick] (3,3.85,5) -- (4, 4.3, 43/9);

\draw [fill = black, opacity = .6](3,6,3.5) -- (3,7.3,4.07) -- (3,6.65,4.5) -- cycle;
 \fill (3,6,3.5) circle (2pt);
  \fill (3,7.3,4.07) circle (2pt);  
\fill (3,6.65,4.5) circle (2pt);
\node[right] at (3, 7, 4.4) {$0$}; 
\node[below] at (3, 6.65, 3.8) {$0$}; 
\node[left] at (3, 6.3, 4.1) {$0$}; 
\fill (0,0, 0) circle(2pt);
\draw[->,thick] (3, 5.7,3.85) -- (0, .3, .2);

\end{tikzpicture}
$$

\caption{As explained in Section~4.1, the infinite solid cone, which is the region 
$\{x\geq y\geq z\geq 0\}$,  is a fundamental domain $F$ for unlabeled networks with three nodes.  
 With the convention that the bottom side of the triangle has weight $x$, the left side has weight $y$, and the right side has weight $z$, the network with edge weights $1, 2, 3$ corresponds to the point $P$ in the interior of the cone. Other networks shown are color coded to correspond to points on faces or edges of the cone.}
\end{figure}
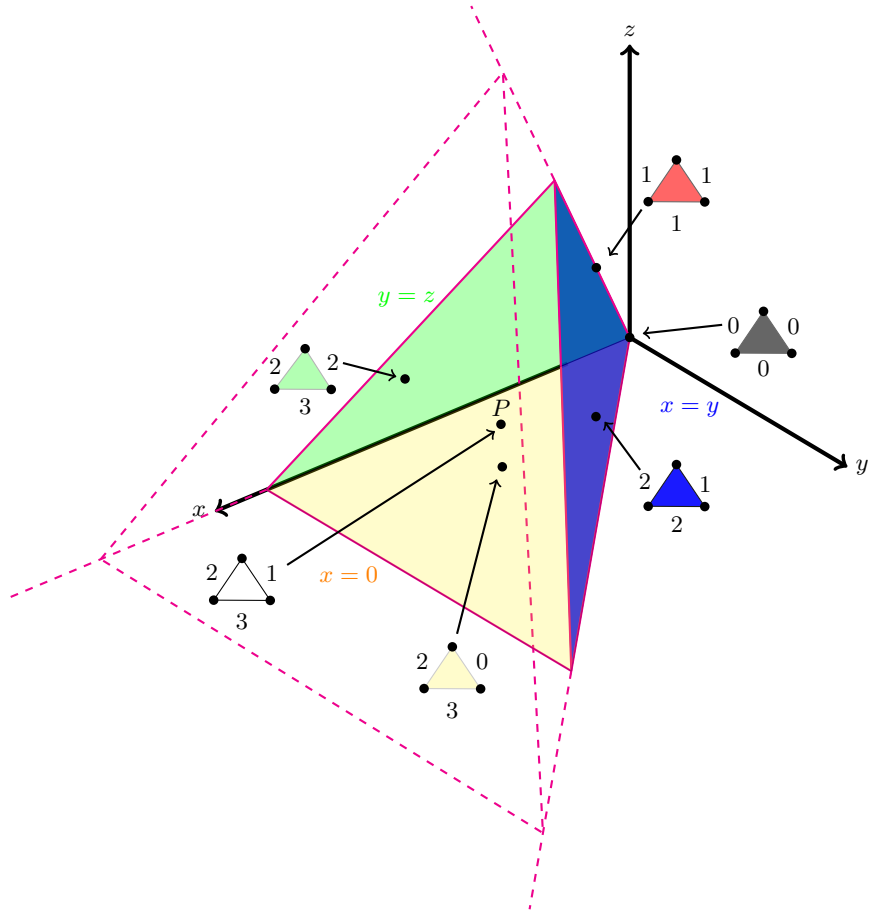


\end{example}

\medskip
The situation is more complicated for graphs with 4 (or more) vertices.  For $d=4$, if we label the edges as $(1,2),\ldots, (3,4)$, then the weight vectors 
$(1,1,1,0,0,0)$ and $(1,1,0,1,0,0)$ have the same distributions of ones and zeros, but correspond to 
binary graphs which are not in the same orbit of $\Sigma_4$.  In particular, the region 
$\{(x_1,\ldots,x_6): x_1\geq x_2\geq \ldots \geq x_6\}$ is not a fundamental domain for the action of 
$\Sigma_4.$

While a fundamental domain is harder to find in high dimensions (see Section~4), 
the overall structure of $\calU_d$ for general $d$ is similar to the $d=3$ case, with just increased notation.

 \begin{theorem}\label{stratspace} The space of unlabeled graphs $\calU_d = \calG/\Sigma_d = \calO_D/\Sigma_d$ is a stratified space. \end{theorem}

The proof is sketched in Supplement A.
By \cite{Lange}, $\R^D/\Sigma_d$ is PL or Lipschitz homeomorphic to $\R^{D-1}\times \R_{\geq 0}$, but the 
 proof does not give a cell decomposition of $\R^D/\Sigma_d$, much less of $\calU_d.$  We do have information about the topology of $\ud$:  in Supplement A we prove that $\ud$ is contractible, and more surprisingly, that the natural slice of $\ud$
 given by the hyperplane $\sum_{i=1}^D x_i = 1$ is contractible.  The practical implication of these results is that the usual topological invariants of $\ud$ and its slice (the fundamental group, the homology/cohomology groups) provide no information.

\section{Network averages and their asymptotic behavior}
\label{sec:net.aves}

In this section we define the mean of a distribution $Q$ on the space of networks and investigate the asymptotic behavior of the empirical (or sample) mean network based on an i.i.d sample of networks from $Q$. Statistical inference  can be carried  out based on the asymptotic distribution of the empirical mean.  We illustrate with an example from hypothesis testing.  The results of the previous section, characterizing the topology and geometry of the space of unlabeled networks, are essential for  achieving our goals in this section. 


\subsection{Network averages through Fr{\'e}chet means}  


Let $Q$ be some distribution on a general metric space $(M,\rho)$.  One can define the Fr\'echet function  $f(p)$ on $M$ as 
\begin{equation}
\label{eq-frechet}
	f(p) = \int_M \rho^2(p,z) Q(dz) \;  (p \in M).
\end{equation}
If $f$ is finite on $M$ and has a unique minimizer 
\begin{align}
\mu = \argmin_p f(p), 
\end{align}
then $\mu$ is called the \emph{Fr\'echet mean} of $Q$ (with respect to the metric $\rho$). Otherwise, the minimizers of the Fr\'echet function form a \emph{Fr\'echet mean set $C_Q$}.   Given an i.i.d sample $X_1,\ldots, X_n\sim Q$ on $M$, the empirical Fr\'echet mean can be defined by replacing $Q$ with the empirical distribution $Q_n= \dfrac{1}{n}\sum_{i=1}^n \delta_{X_i}(\cdot)$, that is,
\begin{align}
\mu_n=\argmin_p \frac{1}{n}\sum_{i=1}^n\rho^2(p, X_i).
\end{align}

When $M$ is a manifold, one can equip $M$ with a metric space structure  through an embedding into some Euclidean space or employing a Riemannian structure of $M$.   Respectively, $\rho$ can be taken to be the Euclidean distance after embedding (extrinsic distance) or the geodesic distance (intrinsic distance), giving rise to extrinsic and intrinsic means. Asymptotic theory for extrinsic and intrinsic analysis has been developed in  \\\cite{rabibook},\\
\cite{ rabivic03}, \\
\cite{rabivic05}, and applied to many manifolds of interest (see e.g., \cite{abs2},\\
 \cite{dryden2009}). 


Now take $M=\calU_d$, the space of unlabeled networks  with $d$ nodes, our space of  interest, and let $Q$ be a distribution on $\calU_d$. 
Given an i.i.d sample $X_1,\ldots, X_n$ from $Q$, in order to define the Fr\'echet mean $\mu$ of $Q$ and empirical Fr\'echet mean  $\mu_n$ of $Q_n$, one needs an appropriate choice of distance on $\calU_d$.  Given the quotient space structure characterized in the previous section, i.e., $\calU_d=\calG_d/\Sigma_d$, a natural choice for the distance $\rho$ 
is the \emph{Procrustean distance} $d_P$, where
\begin{align}
\label{eq-pro}
 d_P([\vec x],[\vec y]) := \min_{\sigma_1,\sigma_2\in \Sigma_d} d_{E}(\sigma_1\cdot\vec x,
 \sigma_2\cdot \vec y),
\end{align}
for  unlabeled networks $[\vec x],[\vec y]\in \calU_d$, with $\vec x$ denoting the vectorized representation of  a representative network $x$. We recall that $\calG_d$ is the set of all labeled graphs with $d$ vertices and $\Sigma_d$ is the group of permutations of $\{1,2,\ldots,d\}$.

In order to carry out statistical inference based on $\mu_n$, defined with respect to the distance \eqref{eq-pro}, some natural and fundamental questions  related to $\mu$ and $\mu_n$  need to be addressed, which we aim to do in the following subsections. Here are some of the most crucial ones: 
\begin{itemize}
\item[1.] (Consistency.) What are the consistency properties of the network empirical mean $\mu_n$, i.e., is $\mu_n$ a consistent estimator of the population Fr\'echet mean $\mu$? Can we establish some notion of a law of large numbers for $\mu_n$?
\item[2.] (Uniqueness of Fr\'echet mean.) This question is concerned with establishing general conditions on $Q$ for uniqueness of the Fr\'echet mean $\mu$.  In general this a challenging task -- indeed,  the lack of general uniqueness conditions for Fr\'echet means is still one of the  main hurdles  for carrying out intrinsic analysis on manifolds \cite{Kendall90}.  To date the most general results in the literature for generic manifolds~\cite{afsari11} force the support of $Q$ to be a  small geodesic ball to guarantee  uniqueness of the intrinsic Fr\'echet mean.  We address this question for the space of unlabeled networks in Section \ref{sec-uniqueness}. 
\item [3.] (CLT.) Once  conditions for uniqueness of $\mu$ are provided, the next key question is whether one can derive the limiting distribution for $\mu_n$ for purposes of statistical inference,  e.g., proving a central limit type of theorem for $\mu_n$, which in turn might be used for hypothesis testing. 
\end{itemize}

We first illustrate the difficult nature of these problems (in particular for question 2 above) through the  example $C=\mathbb{R}^{2}/\mathbb{Z}_{4}$   in Section \ref{eq-sec2}, by explicitly constructing a distribution on $C$ that has non-unique Fr\'echet means.

\begin{example}[Example $C=\mathbb{R}^{2}/\mathbb{Z}_{4}$  continued]
As  in Example \ref{ex-r2modz4} in Section \ref{eq-sec2}, the quotient $C=\mathbb{R}^{2}/\mathbb{Z}_{4}$ is a cone.
Working in polar coordinates and taking $F_{0} = \left[0,\infty \right) \times \left[ -\pi/4,\pi/4 \right)$ to be a fundamental domain, we consider probability distributions of the form $\nu(r,\theta)=\frac{1}{Z}R(r)\chi(\theta)$, where $Z
= \int_{F_0} R(r)\chi(\theta) dr\ d\theta.$ 


We can explicitly compute the Fr\'{e}chet function $f(x)$ with respect to $\nu$.
For $x=(r,\theta) \in F_{0}$,  $F_{\theta} = \left[0,\infty \right) \times \left[ -\pi/4 + \theta, \pi/4 + \theta \right)$ is a fundamental domain. For $y \in F_{\theta}$, $d_{P}(\left[x\right],\left[y\right]) = d_{E}(x,y)$. Then

\begin{align*}
f(x)=\int_{F_{\theta}} ||x-y||^{2} \nu(y)dy=\frac{1}{Z}\left(c_{1}\chi_{1}r^{2}-2c_{2}\chi_{2}(\theta)r+c_{3}\chi_{1}\right),
\end{align*}
where $c_{k}=\int_{0}^{\infty} r^{k} R(r) dr$ for $ k=1,2,3$,  $\chi_{1}(\theta)=\int_{\theta - \pi/4}^{\theta + \pi/4} \chi(t) dt$ and $\chi_{2}(\theta)=\int_{\theta - \pi/4}^{\theta + \pi/4} \chi(t)\cos(\theta - t) dt$.

The Fr\'echet mean occurs when $\partial f/\partial r = \partial f/\partial\theta = 0$, which is difficult to compute in general.  Consider the special case
\begin{equation}
\nu(r,\theta) = \frac{1}{Z}\exp(-(r-\alpha)^2) \enskip ,
\label{eq:simple.nu}
\end{equation}
where $\alpha$ is a fixed constant, $\chi(\theta) \equiv 1$, $Z= (\pi^{3/2}/4)(1 + {\rm erf}(\alpha))$; this distribution for $\alpha = 15$ is plotted in Figure 3. The minimum for this $f$ occurs at 
$$r_{0}=\frac{\sqrt{2} \left(2 \alpha +\sqrt{\pi } e^{\alpha ^2} \left(2 \alpha ^2+1\right) (\text{erf}(\alpha )+1)\right)}{\pi ^{3/2} e^{\alpha ^2} \alpha  (\text{erf}(\alpha )+1)+\pi },$$  with $\theta$ arbitrary.  When $\alpha$ is large, $r_{0} \approx \frac{2\sqrt{2}}{\pi}\alpha$. For $\alpha=15$, $r_{0} \approx 13.5348$. This shows that $\nu$ has a circle's worth of Fr\'{e}chet means; the $\theta$-independence of $\nu$ implies $\theta$-independence of the Fr\'echet means. One can see this in Figure \ref{fig-4} where the Fr\'echet function $f(x)$ is minimal and most blue on a circle of radius approximately 13.53 (corresponding to the red circle on the cone in Figure \ref{fig-3}).

\begin{figure}[h]
\centering
\begin{minipage}{.5\textwidth}
  \centering
  \includegraphics[width=.8\linewidth]{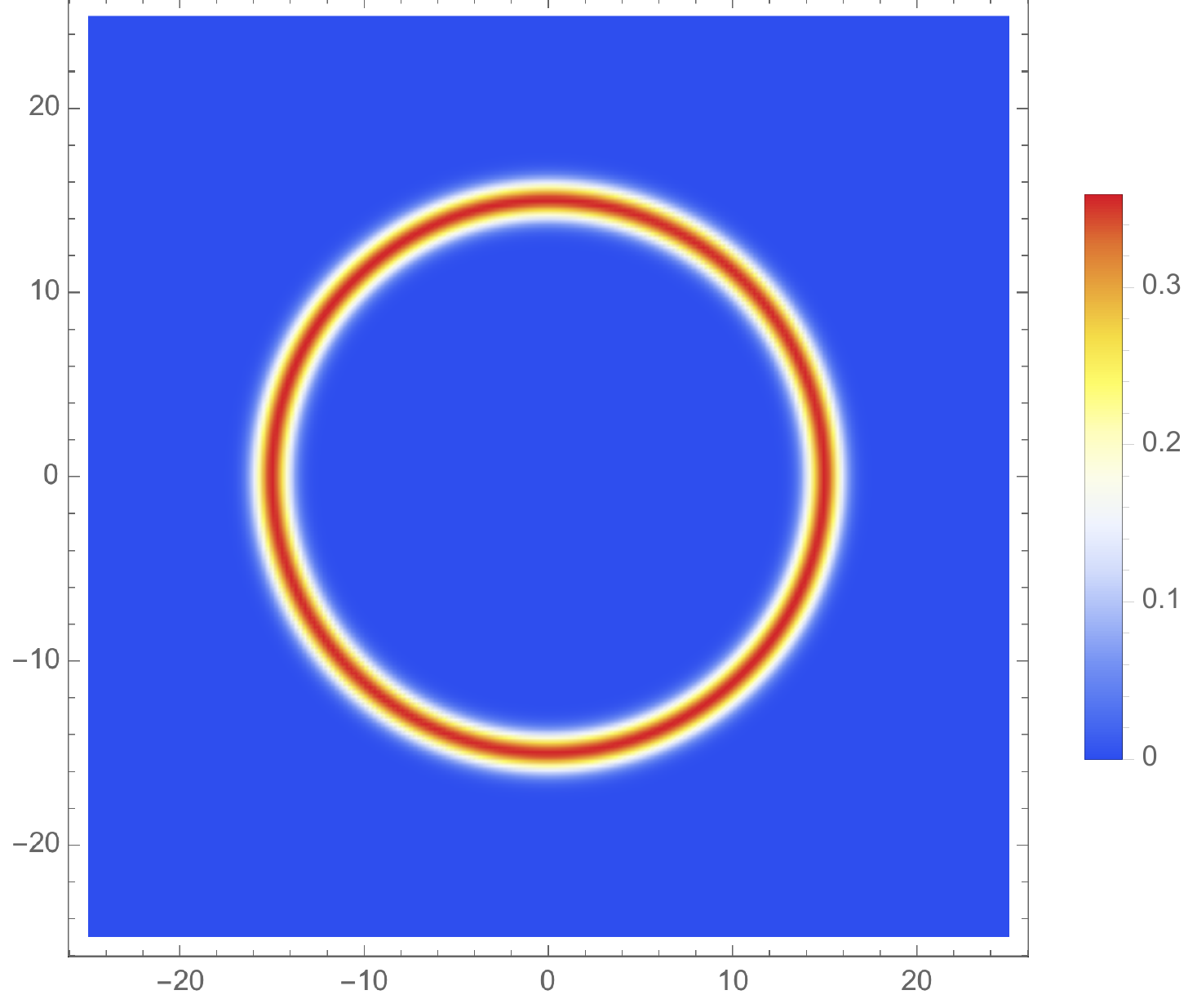}\\
  \label{fig:non-unique_dist_6}
\end{minipage}%
\begin{minipage}{.5\textwidth}
  \centering
  \includegraphics[width=.8\linewidth]{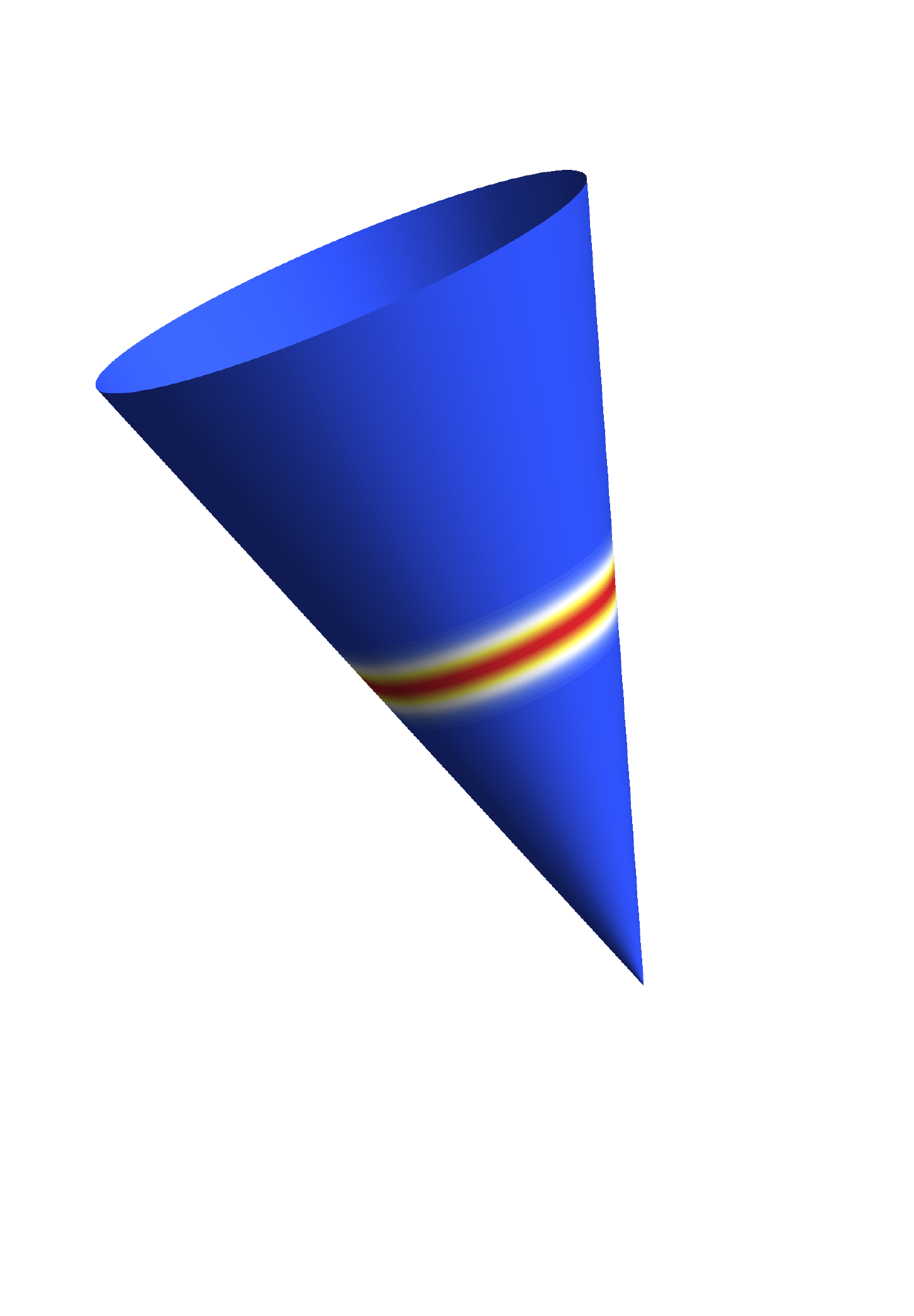}
  \label{fig:non-unique_dist_7}
\end{minipage}
\caption{Plot of the probability distribution $\nu(r,\theta)$ in (\ref{eq:simple.nu}), with $\alpha=15$.  The distribution peaks in the red 
 region where $\nu$ is large, and is small in the blue 
 region. The left hand side shows the plot on $\R^2$, and the right hand side shows that plot on the cone.}
\label{fig-3}
\end{figure}

\begin{figure}[h]
\centering
\begin{minipage}{.5\textwidth}
  \centering
  \includegraphics[width=.8\linewidth]{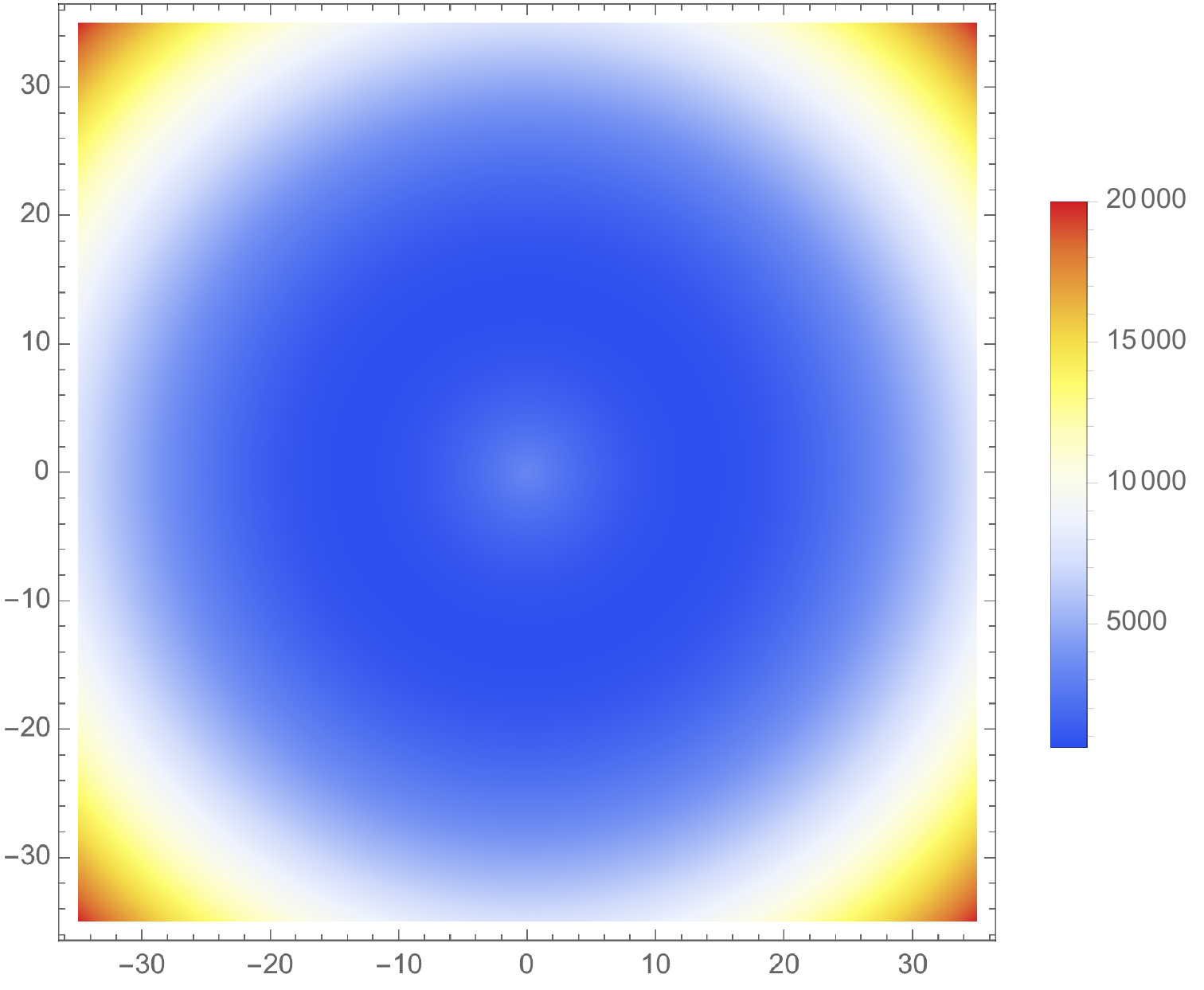}
  \label{fig:non-unique_dist_8}
\end{minipage}%
\begin{minipage}{.5\textwidth}
  \centering
  \includegraphics[width=.8\linewidth]{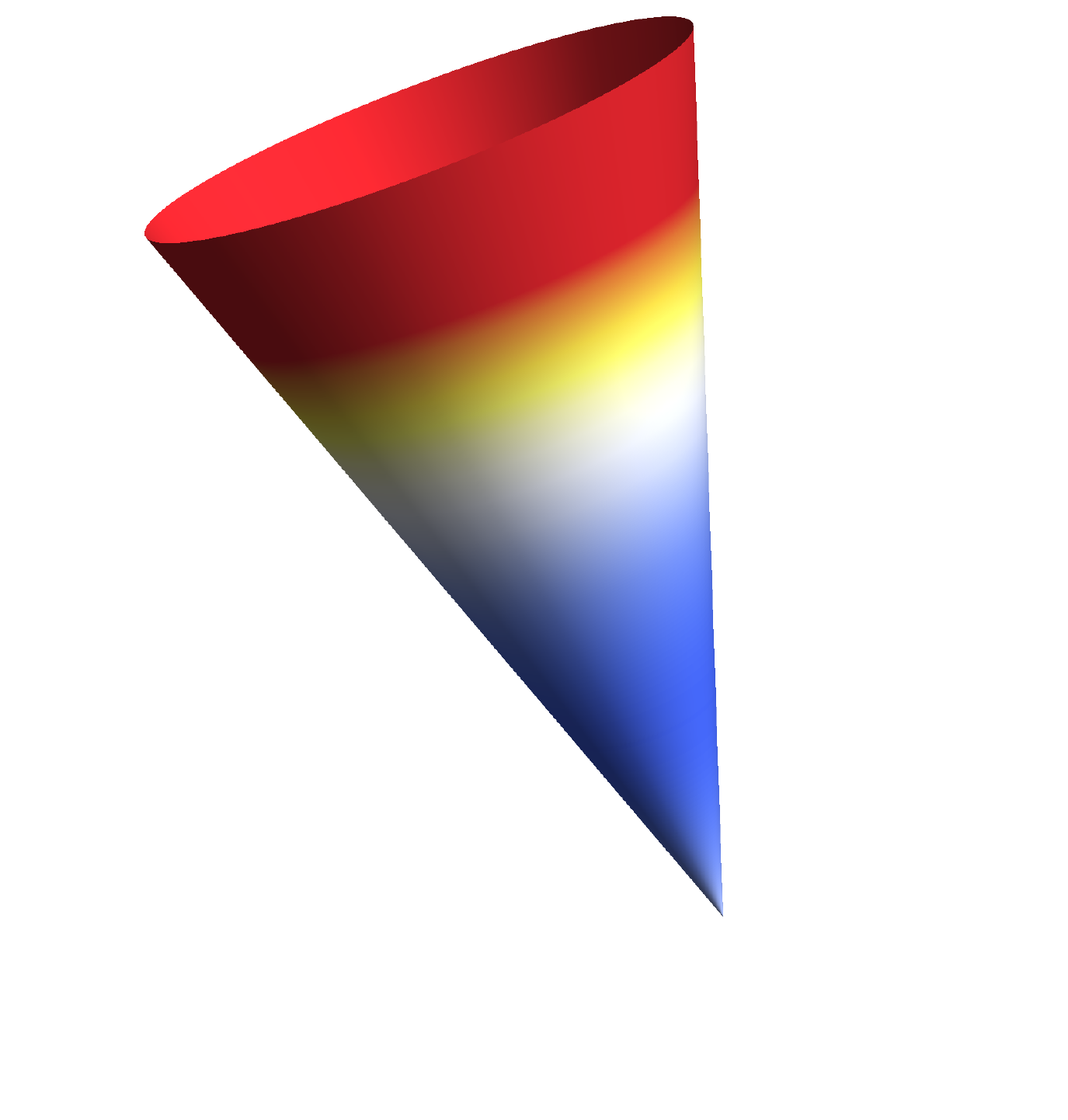}
  \label{fig:non-unique_dist_9}
\end{minipage}
\caption{Plot of the Fr\'{e}chet function $f(x)$ with $\alpha=15$ on $\R^2$ and on the cone. The Fr\'echet mean occurs in the ``most blue" region, on a circle of radius approximately 13.53.}
\label{fig-4}
\end{figure}
 \end{example}
 
%

\subsection{A strong law of large numbers}

Before establishing the limiting distribution for $\mu_n$, a natural first step is to explore the consistency properties of  $\mu_n$.
Drawing on Theorem 3.3 in \cite{rabibook} for general metric spaces, we have the following result.

 \begin{theorem}
 \label{th-consistency}
 Let $Q$ be a distribution on $\calU_d$, let $C_{Q}$ be the set of means of $Q$ with respect to the Procrustean distance $d_P$, and let $C_{Q_n}$ be the set of empirical means with respect to a sample of unlabeled networks $X_1,\ldots, X_n$. Assume that the Fr\'echet function is everywhere finite. Then the following holds:
 (a) the Fr\'echet mean set $C_Q$ is nonempty and compact; (b) for any $\epsilon>0$, there exists a positive integer-valued random variable $N\equiv N( \epsilon)$ and a $Q$-null set $\Omega(\epsilon)$ such that 
\begin{equation}
 C_{Q_n} \subseteq C_Q^\epsilon \doteq \{ p \in \calU_d : d_P(p, C_Q) < \epsilon \} \; \forall\ n \ge N
\end{equation}
outside of $\Omega(\epsilon)$;
(c) if $C_{Q}$ is a singleton, i.e., the Fr\'echet mean $\mu$ is unique, then  $\mu_n$ converges to $\mu$ almost surely. 
 \end{theorem}

 \begin{proof}
 
 We first  prove that every closed and bounded subset of $M=\mathcal U_d$ is compact. 
 
 Let  $F$ be a fundamental domain for the action of $\Sigma_d$ on $\calG_d$, as defined in Section~4, with the associated 
 projection $q:F\to \calU_d$.   This map is continuous and a diffeomorphism on the interior of $F.$  Take a closed and bounded set $S$ in $\mathcal U_d$.  Because $q$ is continuous, $q^{-1}(S) $ is closed. 
 We now show that  $q^{-1}(S)$ is also bounded.   $S$ is contained in a ball centered at some $[\vec z]\in \mathcal U_d$ with radius $r$, so $d_P([\vec z],[\vec x]) < r$ for $[\vx]\in S$.  Now say that the largest entry in $[\vec z]$ (in any ordering of the entries of $\vec z$) is $C.$  If the largest entry in $\vx $ (in any ordering) is 
 greater than $C + r$, then $d_P([\vec z], [\vx]) > (C+ r) - C$, a contradiction.  (This holds because under any permutation $\sigma$ of the entries of $\vx$, one entry in $d_E(\vec z,\sigma  \cdot \vx)$ is greater than $(C+r) - C$.)
 Thus for any choice of $\vx\in [\vx]\in S$, $\Vert \vx\Vert \leq \sqrt{D}(C+r).$  
Thus $q^{-1}(S)$ is contained in the ball of radius $\sqrt{D}(C+r)$ centered at the origin, and thus is bounded.

Since $F$ is a closed subset of $\mathbb R^D$, the closed and bounded set $q^{-1}(S)$ is compact.  Since $q$ is continuous, $S $ is compact.

Then by Theorem 3.3 in \cite{rabibook}, (a) and (b) follow. 

Part (c) follows from \cite{ziezold} under the uniqueness of Fr\'echet mean. 
 \end{proof}
 
 \begin{remark}
 For every model 
 \begin{align}
 Q\in  \mathcal Q=\left\{Q: C_Q\; \text{is a singleton with
  a finite Fr\'echet function}\right \},
  \end{align} the sample Fr\'echet mean $\mu_n$ is a strongly consistent estimate of $\mu$ with respect to this model. 
 \end{remark}

\subsection{A central limit theorem}

The goal of this section is to derive a central limit theorem for the empirical Fr\'echet mean, as an important precursor for statistical inference. One of the key challenges is to
establish  geometric conditions on  distributions on $\calU_d$ which ensure the uniqueness of the population Fr\'echet mean.  We discuss and address the uniqueness issue in detail in Section~4.  Here, our central limit theorem  assumes that the uniqueness conditions of Section~4 are met. 

Let $q:\calO_D\to \calU_d$ be the projection from the space of labeled networks to the space of unlabeled networks.

\begin{theorem}\label{clt}
Assume $Q'$ has support on a compact set $K'\subset \calO_D$ defined in 
Theorem \ref{bestestimate}, so that the pushdown measure $Q =q_*Q'$ supported on 
$K = q(K')$ has a unique Fr\'echet mean $\mu$.  Let $\mu_n$ be the empirical Fr\'echet mean of an i.i.d sample $X_1,\ldots, X_n\sim Q$ with respect to the distance \eqref{eq-pro}. Let $\phi=q^{-1}$. Then we have 
\begin{align}
\label{eq-clt}
\sqrt{n}\left( \phi(\mu_n)-\phi(\mu)\right)\xrightarrow{L} N(0,\Sigma),
\end{align}
where $\Sigma=\Lambda^{-1}C\Lambda^{-1}$ with the Hessian matrix
$$\Lambda=\left({\mathbb E}[D_{r,s}\|\phi(\mu)-\phi(X_1)\|]
 \right)_{r,s=1,\ldots, D},$$
  and $C$ is the covariance matrix of $\{D_r \|\phi(\mu)-\phi(X_1)\|\}_{r=1,\ldots, D}$. 
  \end{theorem}
  
Here  $D_r$ denotes the partial derivative with respect to the $r^{\rm th}$ direction, $D_{r,s}$ denotes second partial derivatives, and $\xrightarrow{L}$ means convergence in law or distribution.

\begin{proof}
Since $\mu_n$ converges to $\mu$ almost surely under our support condition for $Q'$, one can find a small open neighborhood $U$ of $\mu$ inside $K$ such as that $P(\mu_n\in U)\rightarrow 1$. Let $S=q^{-1}(U)$ which is an open subset of $\mathbb R^D$.  Note that $\phi: U\rightarrow S$ is a homeomorphism. By Theorem
\ref{bestestimate}, the projection map $q$ is a Euclidean isometry. Therefore, for any 
vectorized network $\vec x \in S$ and $[\vec z]\in U$,
one has
\begin{align*}
d_P^2([\vx], [\vec z]) = d_P^2(\phi^{-1}(\vec x), [\vec z])  = \|\vec x-\vec z\|^2,
\end{align*}
where $\vec z = \phi([\vec z]).$ 
Thus $d_P^2(\phi^{-1}(\vec x), [\vec z])$ is twice differentiable in $\vec x$ for any $[\vec z]\in U$. Tracing through the definition of the smooth structure on $U$ induced from the standard structure on $\R^D$, we see that $d_P^2([\vec x], [\vec z])$ is twice differentiable in $[\vec x].$
One can also verify the conditions (A5) and (A6) on the Hessian matrices of Theorem 2.2 in \cite{linclt}, and our Theorem follows. 
\end{proof}

As an immediate consequence of this central limit theorem, we can
define natural analogues of classical hypothesis tests.  For example, consider the construction of a statistical test for two or more independent samples using the same framework. Assume that we have $k$ independent sets of networks on $d$ vertices, and consider the problem of testing whether or not these sets have in fact been drawn from the same population. Formally, we have independent samples $X_{i j} \sim Q_j$, for $i=1,\ldots, n_j$ and $j=1,\ldots, k$.  Each of these $k$ populations has an unknown mean,
denoted $\mu^{(j)}$.  Then, as a direct corollary to Theorem \ref{clt}, we have the following asymptotic result. 
\begin{corollary}\label{cor:k-sample}
Assume that the distributions $Q_1,\ldots, Q_k$ satisfy the conditions of Theorem~\ref{clt}.  Moreover, also assume that $n_{j}/n\to p_{j}$ for every sample, with $n:=\sum_{j}n_{j}$, and $0<p_{j}<1$. Then, under $H_{0}:\phi(\mu^{(1)})=\ldots=\phi(\mu^{(k)})$, we have
   \begin{equation}\notag
     T_{k}:=\sum_{j=1}^{k} n_{j}
    \left(\phi(\mu_{j,n_j})-\phi(\mu_n)\right)'
     \widehat{\Xi}^{-1}\left(\phi(\mu_{j,n_j})-\phi(\mu_n)\right) 
     \longrightarrow \chi^{2}_{(k-1)D},
   \end{equation}
   where $\mu_{j,n_j}$ denotes the empirical mean of the $j^{th}$ sample, 
   $\mu_n$ represents the grand empirical mean of the full sample,
   and $\widehat{\Xi} := \sum_{j=1}^{k} \widehat{\Xi}_{j}/n_{j}$ is a
   pooled estimate of covariance, with the $\widehat{\Xi}_{j}$'s denoting
   the individual  covariance matrices estimates of each subsample. 
\end{corollary}

As previously noted, this central limit theorem and such corollaries hold only if the  population 
Fr\'echet mean(s) 
is unique. This depends crucially on  the nontrivial geometry of the space of unlabeled networks. The following section deals exclusively with this issue. 

%
%

%

\section{Geometric requirements for uniqueness of the Fr\'echet mean}
\label{sec-uniqueness}

Underlying the central limit theorem in Theorem~\ref{clt}, the basic question is: which compact subsets $K$ of $\calU_d$ have a unique Fr\'echet mean?
We have seen in Section~\ref{eq-sec2} that $\calU_d$ may be difficult to work with, while a fundamental domain $F\subset \calO_D$ for the action of $\Sigma_d$ on the space of labeled networks $\calO_D$ seems more tractable.  Indeed, finding the 
Fr\'echet mean for a distribution supported in $F$ is a standard center of mass calculation in Euclidean space.  However, it is not clear that this Fr\'echet mean  in $F$ projects to the 
Fr\'echet mean  in the quotient space $\calU_d = \calO_D/\Sigma_d$, because the metric used to compute Fr\'echet means in $\ud$ is the Procrustean distance, which may or may not equal the Euclidean distance.  

In \S4.1,  we find a fundamental domain $F$ by a standard procedure (Lemma \ref{fd}), and find compact subsets $K'\subset F$ for which the Fr\'echet mean in $\calO_d$ is guaranteed to project to the 
Fr\'echet mean of $K\subset \calU_d$, where $K = q(K')$ is the projection of $K'$ under the quotient/projection map $q: \calO_D\to \calU_d = \calO_D/\Sigma_d$. This is the content of the main result in this section (Theorem \ref{bestestimate}). We also show that this result in our particular setting is an improvement of the best result for general Riemannian manifolds due to Afsari \cite{afsari11} (see Figure 6).  In \S4.2, we generalize Theorem \ref{bestestimate} in two different directions: Theorem \ref{main theorem} allows distributions with small tails outside the compact set $K$ above, and Theorem \ref{embedding} shows that a unique Fr\'echet mean can be found for any compact set $K$ after first embedding $K$ isometrically into a larger Banach space.

\subsection{The main result on the uniqueness of the Fr\'echet mean}

We now discuss the construction of a fundamental domain $F$.  By Definition \ref{2.2}, a fundamental domain is characterized by: 
 (i) every weight vector $x$ 
  can be permuted by some $\sigma\in\Sigma_d$ to a network $\sigma\cdot x$ in $F$, (ii) if $\vec w\in \calO_D$ has $\vec w\in F\cap \sigma F$ for $\sigma\in \Sigma_d, \sigma \neq {\rm Id}$, then $\vec w\in \partial F.$  
 (As a technical note, we always consider $\partial F$ with respect to the induced topology on $\calO_D$ from the standard topology on $\R^D.$ ) Once $F$ has been constructed, we are guaranteed that the projection map $q:\calO_D\to \calU_d, \;q(\ve) = [ \ve]$ restricts to a surjective map $q:F\to \calU_d$ which is a homeomorphism from the interior $F^o$ of $F$ to 
 $q(F^o).$  (In fact, $q$ is a diffeomorphism on this region by definition of the smooth structure on
 $q(F^o).$)
  
  \medskip
  It is convenient to center a choice of fundamental domain on a weight vector 
  with trivial stabilizer. 
  
  \begin{definition}  A vector $\vec w = (w_1,\ldots, w_D)\in \calO_D$ is {\em distinct} if it has 
  trivial stabilizer for the action of $\Sigma_d$: {\em i.e.}, 
if $\sigma\cdot \ve \neq \ve$ for all $\sigma\in\Sigma_d$, $\sigma \neq {\rm id}.$
\end{definition}

A vector with trivial stabilizer is also called a vector with trivial automorphism group.  Weight vectors with all different entries  are distinct, which implies that the distinct vectors are dense in $\calO_D.$

For example, consider the two graphs $G_1$ and $G_2$ in Figure~\ref{fig:distinct}.  Both share the same connectivity pattern (i.e., are isomorphic), but have different weight vectors.  The weight vector $\ve^1= (20,20,7,\ldots)$ of $G_1$ satisfies 
$\sigma\cdot\ve^1 = \ve^1$ for $\sigma = (23)\in\Sigma_7$.
In contrast, for all $\sigma\in\Sigma_7\setminus \{{\rm id}\}$,
$\sigma\cdot\ve^2\neq \ve^2,$
where $\ve^2$ is the weight vector of $G_2$,
because the two 20's belong to nodes with different valences.   
Thus  even though $\ve^1, \ve^2$ have the same set of weights, 
$\ve^2$ is distinct, while $\ve^1$ is not.

\begin{figure}
\centering
$$
\begin{tikzpicture}
\node[left] at (0,.5) {{\bf $G_1$}};
 \draw (2,1) -- (5,1);
 \fill (2,1) circle (2pt);
  \fill (5,1) circle (2pt);
 \node[above] at (2,1) {$1$};
 \node[left,red] at  (1.4,.5) {$20$};
  \node[above] at (5,1) {$4$};
   \node[above,red] at  (3.5,1) {$7$};
  \draw (2,1) -- (1,0);
  \fill (1,0) circle (2pt);
\node[below] at (1,0) {$2$};
 \draw (2,1) -- (3,0);
 \fill (3,0) circle (2pt);
\node[below] at (3,0) {$3$};
 \node[right,red] at  (2.6,.5) {$20$};
\draw (5,1) -- (4,0);
\fill (4,0) circle (2pt);
\node[below] at (4,0) {$5$};
 \node[left,red] at  (4.4,.5) {$8$};
\draw (5,1) -- (5,0);
\fill (5,0) circle (2pt);
\node[below] at (5,0) {$6$};
 \node[left,red] at  (5,.5) {$9$};
\draw (5,1) -- (6,0);
\fill (6,0) circle (2pt);
\node[below] at (6,0) {$7$};
 \node[right,red] at  (5.5,.5) {$10$};
\end{tikzpicture}  
$$

$$
\begin{tikzpicture}
\node[left] at (0,.5) {{\bf $G_2$}};
 \draw (2,1) -- (5,1);
 \fill (2,1) circle (2pt);
  \fill (5,1) circle (2pt);
 \node[above] at (2,1) {$1$};
 \node[left,red] at  (1.4,.5) {$20$};
  \node[above] at (5,1) {$4$};
   \node[above,red] at  (3.5,1) {$7$};
  \draw (2,1) -- (1,0);
  \fill (1,0) circle (2pt);
\node[below] at (1,0) {$2$};
 \draw (2,1) -- (3,0);
 \fill (3,0) circle (2pt);
\node[below] at (3,0) {$3$};
 \node[right,red] at  (2.6,.5) {$10$};
\draw (5,1) -- (4,0);
\fill (4,0) circle (2pt);
\node[below] at (4,0) {$5$};
 \node[left,red] at  (4.4,.5) {$8$};
\draw (5,1) -- (5,0);
\fill (5,0) circle (2pt);
\node[below] at (5,0) {$6$};
 \node[left,red] at  (5,.5) {$9$};
\draw (5,1) -- (6,0);
\fill (6,0) circle (2pt);
\node[below] at (6,0) {$7$};
 \node[right,red] at  (5.5,.5) {$20$};
\end{tikzpicture}  
$$
\caption{Two networks of identical connectivity with ($G_1$) and without ($G_2$) distinct weight vectors. \label{fig:distinct}}
\end{figure}
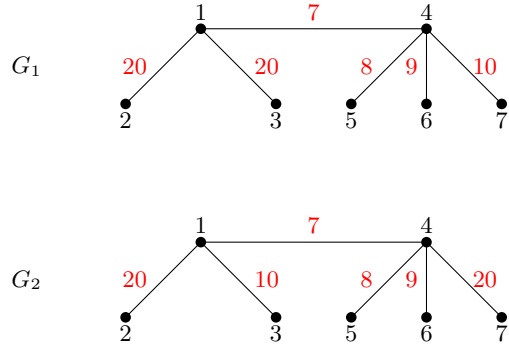

We now explain a standard procedure to construct a fundamental domain as one region in the Voronoi diagram of the orbit of a distinct vector. (Geometers call this a Dirichlet domain.)
Let $d_P$ be the Procrustean distance on $\un.$  From now on, we just write $  \vec w$ instead of
$[  \vec w]$ for elements of $\ud.$

\begin{lemma}\label{fd}  Fix a  distinct
vector $  \vec w\in \calO_D$.  Set 
\begin{align*}F &= F_{\vec w} = \{\vec w'\in \calO_D: d_E(\vec w,\vec w') \leq d_E(\vec w,\sigma\cdot \vec w'),
\ \forall\sigma\in\Sigma_d\} \\
&= \{\vec  w'\in \calO_D: d_E(\vec w,\vec w') = d_P
(\vec w, \vec w')\}.
\end{align*}
Then 

(i) $F$ is a fundamental domain for the action of $\Sigma_d$ on $\calG_d$.

(ii) $F$ is a solid cone with polyhedral cross section. 
\end{lemma} 

In the proof,  we use the fact that $\vec w$ is distinct just below (\ref{hyperplane}).

\begin{proof}   (i) First, for fixed $\vec w_1$, a minimum of $d_E(\vec w,\sigma\cdot \vec w_1)$ is attained, since $\Sigma_d$ is finite.  Thus every network (characterized by its weight vector
$\ve_1$) has a permutation in $F$.  

Second, we can rewrite $F$ as
$$F =  \{\vec  w_1\in \calO_D: d_E(\vec w,\vec w_1) \leq d_E(\sigma\cdot\vec w,\vec w_1),
\ \forall\sigma\in\Sigma_d\}.$$
Let $\ve_1\in F\cap\sigma F$.  Then 
$\sigma^{-1}\ve_1\in F$, and
\begin{align*} d_E(\ve,\ve_1) &= \min_\tau d_E(\ve,\tau\ve_1), 
\  d_E(\ve,\sigma^{-1}\ve_1) = \min_\tau d(\ve,\tau\sigma^{-1}\ve_1)\\
& = \min_\tau d(\ve,\tau\ve_1).
\end{align*}
Thus
\begin{equation}\label{hyperplane} d_E(\ve,\ve_1) = d_E(\ve, \sigma^{-1}\ve_1) = d_E(\sigma\ve,\ve_1),
\end{equation}
so $\ve_1$ is equidistant to $\ve$ and $\sigma\ve.$   Since $\ve$ is distinct, $\sigma\ve\neq \ve$ for $\sigma \neq {\rm id}.$ Thus $\ve_1$ lies on a hyperplane in $\calO_D$ defined by (\ref{hyperplane}), and any ball around $\ve'$ contains points that are closer to $\ve$ than to $\sigma\ve$, and points that are farther from $\ve$ than from $\sigma\ve.$  Therefore
 $\ve_1\in\partial F.$ 
 
 (ii) We can construct $F'$ for the action of $\Sigma_d$ on all of $\R^D$ by taking the set of hyperplanes $H_\sigma$ of points equidistant from $\ve$ and $\sigma\ve$ for $\sigma\in\Sigma_d$, and taking the connected component of 
$\calO_D\setminus \cup_\sigma H_\sigma$ containing $\ve.$  Since these hyperplanes all pass through the origin, this component is a solid cone on the origin.  The boundary is given by a union of hyperplanes, so the cross section is a polyhedron. (Not all hyperplanes contribute edges to the cross sectional polygon; see the next example.)  Moreover, $\Sigma_d$ preserves $\calO_D$, so a 
fundamental domain in $\calO_D$ is given by $F= F'\cap\calO_D.$ The boundary planes of $F$ are either boundary planes of $F'$ or (subspaces of) the boundary of $\calO_D.$
\end{proof}

This Dirichlet/Voronoi fundamental domain depends on a choice of $\vec w.$  In particular, for a fixed distinct 
$\vec w_0$, we can guarantee that $\vec w_0$ is in the interior of $F$ by setting $\vec w = \vec w_0$ in the lemma.

$F$ is a solid cone cut out by at most $d!-1+D$ hyperplanes, where $d!-1$ is the order of 
$\Sigma_d\setminus \{{\rm Id}\}$ and $D$ is the number of coordinate hyperplanes.  Thus this construction of $F$
 is not very practical except in low dimensions.  Looking back at Figure 2, the infinite solid cone is 
 the fundamental domain for  the distinct vector $(3,2,1)$.
 See Appendix \ref{Jackson2} for an algorithm that computes the fundamental domain and examples for $d=3,4.$

We now give more information about $F$.
The following result, although interesting, is not used below.

\begin{prop}\label{4.6} Let $F$ be the fundamental domain associated to a distinct vector
 $\vec w\in \calO_D.$  All distinct vectors in $F$ 
 have a representative in the interior $F^o$ of $F$.   

\end{prop}

\begin{proof}   See Supplement B.
\end{proof}

\begin{remark}\label{forward remark}  Any nondistinct vector $\ve'$ has an arbitrarily close distinct vector
$\ve$.   Then $d_P({\ve},{\ve'}) = d_E(\ve, \ve')$, and this remains true for vectors close to $\ve.$  
Therefore, $\ve'$ is in the interior $F^o$ of $F = F_{\ve}.$ Thus for any nondistinct vector $\ve'$, we can find a fundamental domain that contains $\ve'$ in its interior.
\end{remark}

The next lemma gives a sense of the minimal size of $F$. 
Let $F_c$ denote the solid cone with vertex at the origin, axis $\ve$, and cone angle $c$. Let $a = a_{\ve}$ be the smallest angle between $\ve$ and $\sigma\cdot\ve$ for $\sigma\in\Sigma_d,$
for  $\sigma\neq {\rm id}.$ Of course, 
$a/2$ is the smallest angle between $\ve$ and a hyperplane boundary of $F$.

\begin{lemma}\label{solidcor} $F$ contains the solid cone $F_{a/2}$.
\end{lemma}


\begin{proof}  
We claim that for vectors $\vu, \vel , \vv$, the angles formed by them satisfy
$$\angle(\vu,\vel ) + \angle (\vel ,\vv) \geq \angle(\vu,\vv).$$
To prove this, we may assume the $\vu,\vv,\vel $ are unit vectors. Let $\vec\ell'$ be the
projection of $\vel $ into the plane of $\vu$ and $\vv.$ Assume that $\angle(\vu,\vel ') + \angle(\vel ',\vv) = \angle(\vu,\vv).$  
Then $\cos(\angle(\vu, \vel ')) = \vu\cdot \vel '/|\vel '| \geq \vu\cdot \vel '$, and $\cos(\angle(\vu, \vel )) = 
\vu\cdot\vel  = \vu\cdot \vel ' + \vu\cdot(\vel -\vel ') = \vu\cdot\vel '$, since $\vel -\vel '$ is normal to the $(\vu, \vv)$-plane.  Thus $\angle(\vu,\vel ') \leq \angle(\vu, \vel ).$  Similarly, $\angle(\vel ',\vv) \leq \angle(\vel ,\vv).$  Thus $\angle(\vu,\vel ) + \angle (\vel ,\vv) \geq \angle(\vu,\vel ') + \angle (\vel ',\vv)
= \angle(\vu,\vv).$  The other possibility is that
$\angle(\vu,\vel ') - \angle(\vel ',\vv) = \angle(\vu,\vv)$ or the same with $\vu$ and $\vv$ switched. In this case, $\cos(\angle(\vu, \vel )) = \vu\cdot\vel ' \leq \cos(\angle(\vu, \vel '))$ implies $\angle(\vu, \vel )
\geq \angle(\vu,\vv).$

Thus for a fixed permutation $\sigma$, we  have
$\angle(\vel ,\vu) + \angle(\vu,\sigma\cdot\vel ) \geq \angle( w, \sigma\cdot  w) \geq a$. Therefore 
\begin{equation}\label{dist} \vu\in F_{a/2}\Longrightarrow \angle(\vu,\sigma\cdot\vel )\geq a/2.
\end{equation}
  We may assume that $|\vel | = |\vu|.$ 
Because $\vu$ and the $\sigma\cdot\vel $ all lie on the sphere of radius $|\vel |$, the distances between $\vu, \vel , \sigma\cdot\vel $ are proportional to the angles they form with the origin.  Thus (\ref{dist}) implies that $d^E(\vu,\sigma\cdot \vel ) \geq d^E(\vu, \vel ),$ so $\vu\in F.$
\end{proof}

Although the topology of the interior $F^o$  and its image $q(F^o)$ are the same, their geometries are very different; this underlies the difference in general between  
the Fr\'echet means in $F^o$ and $q(F^o).$   

\begin{example}[Example $C=\mathbb{R}^{2}/\mathbb{Z}_{4}$  continued]
\label{ex}
$F^o$ is the open first quadrant.  
Let $\vec v,\vel\in F^o$ cut out angles $\alpha,\beta$ with the positive $x$-axis, respectively. By the law of cosines, the Euclidean distance 
between $\vec v, \vel  $ is less than the distance between $\vec v$ and 
$R_{\pi/2}\vel  $, the rotation of $\vel  $ by $\pi/2$ radians counterclockwise,
iff $|\alpha -\beta |< \pi/4.$  Thus for $|\alpha-\beta| > \pi/4$, 
$d_E(\vec v,\vec w) < d_P([\vec v], [\vel ]).$  Thus distances in $F^o$ and $C$ are different.

This affects the Fr\'echet means.  Let $\nu = (4/3\pi) rdrd\theta $ be the uniform probability measure on $F$ supported on $\{(r,\theta)\in [1,2]\times [0, \pi/2]\}.$  
The Fr\'echet mean  on $F$ is the center of mass $(3/2,3/2).$  The cone $C$ has a circle action which rotates points equidistant from the vertex, and the Procrustean distance is clearly invariant under this action. This implies that if 
$[(r_0,\theta_0)]$ is a Fr\'echet mean on $C$, so is $[(r_0,\theta)]$ for all $\theta.$  Therefore, we can compute the Fr\'echet mean at $\theta = \pi/4.$  
Since $|(\pi/4)-\beta| \leq \pi/4$ in $F$, $d_E((r_0,\pi/4), (r,\theta)) = 
d_P((r_0,\pi/4), (r,\theta))$ for all $(r,\theta) \in F.$  The previous computation 
gives $r_0 = 3/2.$  We conclude that the Fr\'echet mean on $C$ is the entire circle $r = 3/2.$

\end{example}







We now find a sub-cone of $F$ such that the Fr\'echet mean of a compact convex set $K$ inside this sub-cone
projects to the unique Fr\'echet mean of the associated quotient space $K = q(K')$ inside $\calU_d.$ As explained at the beginning of this section, this allows us to derive a central limit theorem on $K$. 

\begin{theorem}  \label{bestestimate} 
Let $K\subset\mathcal U_d$ be  such that there exists a compact convex set $K'\subset F_{a/4}$ with $q:K'\to K$ a homeomorphism.  Then the Fr\'echet mean $\mu_{K}$ of $K$ is unique and satisfies $q(\mu_{K'})=\mu_{K}$.
\end{theorem}

See Figure \ref{fig-cone} for a schematic picture. This result is discussed without a full proof in \cite{Jainb2016, Jaina2016}.  Note that 
 $\mu(K)$ can be computed by finding the center of mass $\mu_{K'}$ of $K'$ (with respect to the pullback of a distribution 
$Q$ supported in $K$) by standard integrals, and then projecting to $\ud.$

\begin{proof}
Take $\vu, \vv\in F_{a/4}$ with $|\vu|=|\vv|=1.$   
As in the previous lemma, 
  $\angle(\vu,\vel ) + \angle (\vel ,\vv) \geq \angle(\vu,\vv)$, with equality iff $\vu,\vv,\vel $ are coplanar.  Thus
$$\angle(\vu,\vv) \leq  2(a/4).$$
 On the other hand, 
for $\sigma \neq {\rm id}$, we have as above
$$\angle(\vu,\sigma\cdot\vv) + \angle(\sigma\cdot\vv,\sigma\cdot\vel ) \geq\angle(\vu,\sigma\cdot\vel )
\Rightarrow \angle(\vu,\sigma\cdot\vv) \geq \angle(\vu,\sigma\cdot\vel ) - \angle(\vv,\vel ). $$
Let the plane containing $\vec 0, \vu, \sigma\cdot\vel$ intersect $\partial F$ at a line containing the unit vector
$\vec z$.  Then
$$\angle(\vu,\sigma\cdot\vel ) = \angle(\vu,\vec z) + \angle(\vec z, \sigma\cdot\vel) \geq (a/2-a/4) + a/2 
= 3a/4.$$
Therefore
$$\angle(\vu,\sigma\cdot\vv) \geq 3a/4 - a/4 = a/2 \geq \angle(\vu,\vv).$$
Since $|\vu| = |\vv|= |\sigma\cdot\vv|= 1$, the distance between these vectors is proportional to their angles.  This implies that
$$d_E(\vu,\vv) = d_P([\vu],[\vv])$$
on $F_{a/4}.$

As a compact convex subset of Euclidean space, $K'$ has a unique Fr\'echet mean. It follows that $K$ is compact convex.  
$q:K'\to K$ is a homeomorphism, so for $[\vx]\in K$, the Fr\'echet functions on $K'$ and $K$ satisfy
$$f([\vx]) = \int_{K} d^2_P([\vx], [\vec y])q_*Q(d[\vec y]) = \int_{K'}d^2_E(\vx,\vec y) Q(d\vec y) = f(\vec x).$$
Thus the Fr\'echet mean $\mu_{K'}$ of $f$ on $K'$ 
projects to $\mu_{K}$, the unique Fr\'echet mean of $f$ on $K$.


\end{proof}

From Thm.~\ref{bestestimate}, we derive the main result Thm.~\ref{clt}.  


\begin{figure}
\centering
$$
\begin{tikzpicture}

\draw[fill = yellow] (0,0) -- (2.25,7.5) -- (5,6) --(0,0);
\draw[fill = yellow] (0,0) -- (8,5) -- (5,6) --(0,0);
\draw[fill = yellow] (0,0) -- (8,5) -- (10.5,4.95) --(0,0);

\draw [red] (4.5,-4) rectangle (6,-3.5);
\draw[red] (4.5,-3.5) to (5.2,-3.2) to (6.7,-3.2) to (6,-3.5);
\draw[red] (6.7,-3.2) to (6.7,-3.7) to (6,-4);
 \draw [dashed,red, thick] (5.2,-3.2) to (5.2,-3.7) to (4.5,-4);
  \draw [dashed,red, thick]  (5.2,-3.7) to (6.7,-3.7);
  \node[above] at (5.5, -3.2) {$K$};

\draw [red] (10.3,8.5) rectangle (11.8,9);  
\draw[red] (10.3,9) to (11,9.3) to (12.5,9.3) to (11.8,9);
\draw[red] (12.5,9.3) to (12.5,8.8) to (11.8,8.5);
 \draw [dashed,red, thick] (11,9.3) to (11,8.8) to (10.3,8.5);
  \draw [dashed,red, thick]  (11,8.8) to (12.5,8.8);
  \node[above] at (11.5, 9.4) {$K'$};

\draw[<->, thick] (11,0) -- (0,0) -- (0,11); 

 \draw[thick, rotate = -20, shift = {(1.15,5.15)}, green] (7.5,7,6) ellipse (2.5cm and .2cm);
  \draw[thick, rotate = -18, shift = {(3.85,9.3)}, blue] (7.5,7,6) ellipse (1.98cm and .2cm);

 \node[right] at (11,0) {$x_1$};
\node[above] at (0,11) {$x_{D}$};
\draw[->, thick] (0,0) -- (9.72,8.1);
\node[right] at (9.72,8.1) {$\vel $};
\draw (2.5,3) -- (4,2.5) -- (7, 3.3);
\draw (1.5,5) -- (2.5,3);
\draw[dashed] (7,3.3)  -- (1.5,5);
\draw (0,0) -- (2.25,7.5);
\draw (0,0) -- (8,5);
\draw (0,0) -- (5,6);
\draw (0,0) -- (10.5,4.95);
\draw (2.25,7.5) -- (5,6) -- (8,5) -- (10.5,4.95) -- (2.25,7.5);

\draw[thick,green] (0,0) to (11.7, 6.23);
\draw  (8,6.7) arc [radius=2, start angle=59, end angle= -6];
\node[above] at (9.1,6.2) {$\angle a/2$};
\draw  (9.3,7.7) arc [radius=2, start angle=44, end angle= 80.8];
\node[above] at (9.1,8.2) {$\angle a/3$};
\draw[thick, green] (0,0) to (62.8/9,71.5/9);
\node[right, green] at (7.2,7) {${\bf F}_{a/2}$};

\node[right] at (4.7,6.4) {${\bf F}$};
\draw [->,thick] (5.5,-.5) to (5.5,-2);
\node[right] at (5.5,-1.25) {$q$};

\draw (5.5,-3.5) ellipse (2cm and 1 cm);
\node[right] at (7.8, -3.5) {$\ud$};
\draw (3.5,-3.5) to [out=-20,in=195] (7.5,-3.5);

\draw[blue] (0,0) to (14.83,9.9);
\draw[blue] (0,0) to (11.05,11.15);
\node[blue] at (13.5, 9.8) {${\bf F}_{a/4}$};

\draw (0,-3.6) to (4.82,-2.54);
\draw(0,-3.6) to (4.82, -4.47);
\draw (2,-4) arc (-13:10:2);
\draw[dashed](2.0, -3.9) arc (-160:-200:1);
\end{tikzpicture}
$$
\caption{ $K\subset F_{a/4}$, the blue cone.
$K$ and $K'$ are homeomorphic via $q.$  
The Euclidean distance between points $\vx,\vec y\in F_{a/4}$ is the same as the Procrustean distance between their orbits $[\vec x],[\vec y],$ so $K'$ and $K$ are actually isometric. The Fr\'echet means of $K$ and $K'$ are related by 
$\mu_{K} = q(\mu_{K'}).$  In particular, the Fr\'echet mean of $K$ is unique. }
\label{fig-cone}
\end{figure}
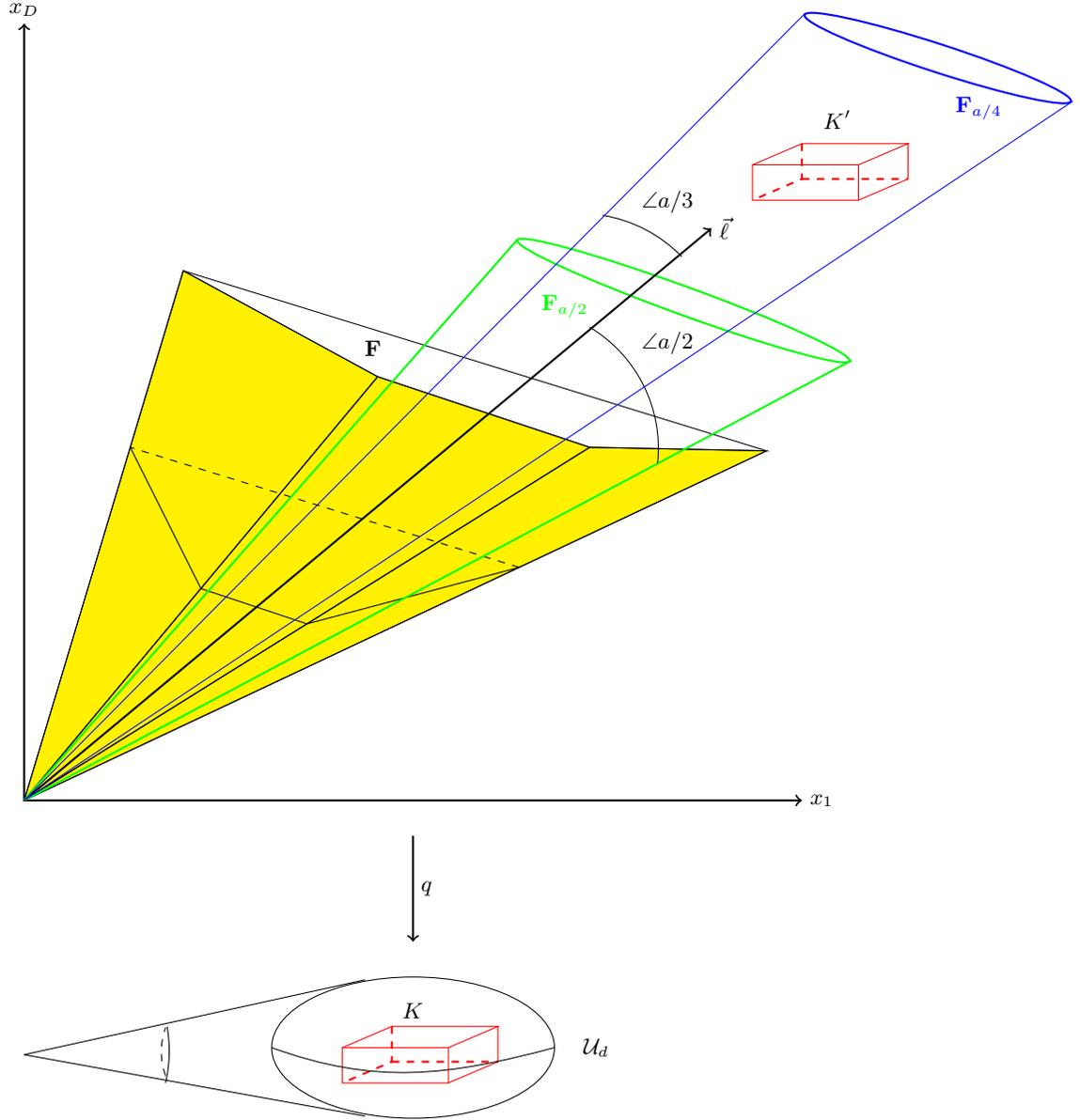


We now prove that Thm.~\ref{bestestimate} is a quantitative improvement over the optimal 
estimate for general Riemannian manifolds due to Afsari:

\begin{theorem}{\rm \cite[Thm.~2.1]{afsari11}  }\label{af}
Let $M$ be a complete Riemannian manifold with sectional curvatures at most $\Delta$ and 
injectivity radius $\iota_M.$  Set 
$$\rho_0 = \frac{1}{2}\min\left\{\iota_M, \frac{\pi}{\Delta}\right\},$$
with the convention that $\frac{\pi}{\Delta} = \infty$ if $\Delta \leq 0.$  For any $\rho<\rho_0$, a 
geodesic ball of radius $\rho$ in $M$ has a unique Fr\'echet mean.
\end{theorem}

The injectivity radius is the supremum of  $r>0$ such that every geodesic ball of 
radius $r$ is a topological ball. Afsari's theorem does not apply  to $\calU_d$, which is not a manifold.  We also cannot apply this theorem to the more tractable interior $F_ w^o$ of $F_ w$, which is diffeomorphic to $\calU_d$ minus a set of measure zero, because  $F_ w^o$ is not complete in the Euclidean metric, and has zero injectivity radius.  

However, we can compare Afsari's result to Thm.~\ref{bestestimate} on (the locally complete) geodesic balls inside $F_a$, as these are ordinary Euclidean balls. The Euclidean metric has zero curvature, so
$\pi/\Delta = \infty$ in our convention.  Therefore, 
we need the injectivity radius of the smooth points of cone $F_a.$  Take $\vv$ lying on the cone axis.  A ball $B_{\vv}$ 
centered at $\vv$ and tangent to the cone at a point $P$ determines a right triangle $\Delta O\vv P$. This ball has radius $|\vv|\sin(a/2)$, so this is the injectivity radius $\iota_{B_{\vv}}$ in 
Thm.~\ref{af}.
Thus Afsari's theorem applies to a ball of half this radius, denoted $B_{\vv}\left(|\vv|\sin(a/2)/2\right).$

For Thm.~\ref{bestestimate}, we can take any compact set $K'$ inside $F_{a/4}.$  To show that this Theorem improves the general Afsari result, we find such a $K'$ containing $B_{\vv}\left(|\vv|\sin(a/2)/2\right).$  This follows if the $F_{a/4}$ cone contains the cone containing $B_{\vv}\left(|\vv|\sin(a/2)/2\right)$, which has cone angle 
$\sin^{-1}(\sin(a/2)/2).$  Since $\sin$ is increasing for $a\in (0,\pi/2)$, it suffices to show that 
\begin{equation}\label{sine} \sin (a/4) \geq \sin(a/2)/2.
\end{equation}
This follows from $\sin(2\theta) = 2\sin(\theta)\cos(\theta) \leq 2\sin(\theta).$

As a result, $K' = q(K)\subset \calU_d$ has a unique Fr\'echet mean, even though its radius is larger than the bound in Thm.~\ref{af}.  This just says that the Afsari bound, which is universal for all Riemannian manifolds, may have an improvement on specific manifolds.
Figure \ref{fig-improvement} illustrates this improvement.  
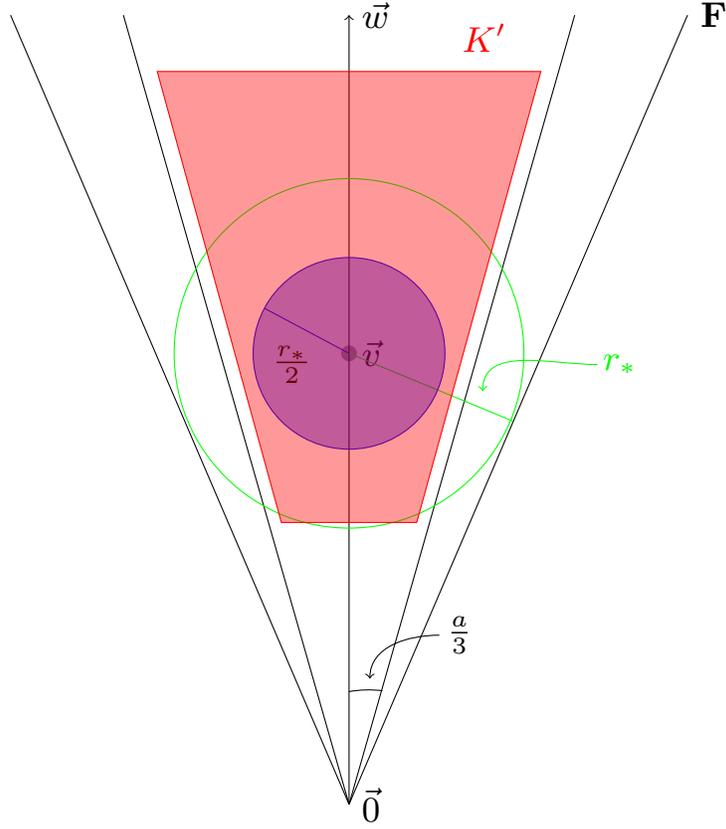
\begin{figure}
\centering
$$
\begin{tikzpicture}[scale=1.5, transform shape, fill opacity = .4]
\draw (0,0) -- (-3,7);
\draw(0,0) -- (3,7);
\draw (0,0) -- (-2,7);
\draw (0,0) -- (2,7);
\node[right,opacity = 1] at (0,0) {$\vec 0$};
\draw [->](0,0) -- (0,7);
\node[right,opacity = 1] at (0,7)  {$\vec  w$};
\node[right,opacity = 1] at (3,7){{\bf F}};

\fill (0,4) circle  (2pt);  
\node[right, opacity = 1] at (0,4) {$\vec v$};
\draw[green] (0,4) -- (1.45,10.2/3);
\draw[blue, fill = blue] (0,4) circle (.85 cm);
\node[green, opacity = 1] at (2.4, 3.9) {$r_*$};
\node[anchor=east] at (1.3,3.55) (line) {};
  \node[anchor=west] at (2.2,3.9) (r) {};
  \draw[<-,green] (r) edge[out=180,in=90,->] (line);
  
\node[right, opacity = 1] at (-.8,3.9) {$\frac{r_*}{2}$};
\draw[blue] (0,4) -- (-.75,4.4);

\draw[green] (0,4) circle (1.55 cm);

\draw (0,1) arc (100:83:1);
\node[right,opacity = 1] at (.75,1.5) {$\frac{a}{3}$};
\node[anchor=west] at (.8,1.5) (arc angle) {};
\node[anchor=east] at (.3,1) (arc itself) {};
\draw[<-] (arc angle) edge[out=180,in=90,->] (arc itself);

\filldraw[draw =red, fill = red] (.6,2.5) -- (1.7,6.5) -- (-1.7,6.5) -- (-.6,2.5) -- (.6,2.5);
\node[red, opacity = 1] at (1.2, 6.8) {$K'$};

\end{tikzpicture}
$$
\caption{A cross-section of the fundamental domain {\bf F} showing the improvement of Thm.~\ref{bestestimate} over the Afsari result. The Afsari work gives a unique 
Fr\'echet mean of the image in
$\calU_d$ to
 the ball of radius $r_*/2$, where $r_*$ is the injectivity radius at $\vv.$  Thm.~\ref{bestestimate} 
 gives a unique Fr\'echet mean to the image in $\calU_d$ of the
larger compact set $K'$ inside the $a/3$ cone.  The boundary of $K'$ can extend down to $\vec 0$, outwards as far as the walls of the $a/3$ cone, and to any finite height. }
\label{fig-improvement}
\end{figure}


\subsection{ Generalizations of Thm.~\ref{bestestimate}}
Thm.~\ref{bestestimate} can be improved in two directions.  The first improvement, Theorem \ref{main theorem}, is a technical analytic extension which shows that
 the uniqueness of the Fr\'echet mean still holds for smooth distributions that are close to compactly supported distributions supported in the $F_{a/4}$ cone.  In practical terms, 
this means that distributions may have small tails outside the $F_{a/4}$ cone and still have a unique Fr\'echet mean.  The second improvement, Theorem \ref{embedding} shows that any compact subset $E\subset \ud$ has a unique Fr\'echet mean after isometrically embedding $E$ into the Banach space of bounded functions on $E$ in the sup norm.  While this is more appealing, the new Fr\'echet mean may not lie in the image of $E$ and may be hard to compute.

\begin{theorem} \label{main theorem}
Let $Q $ be a smooth probability distribution with support inside a compact set $ 
E\subset F_{a/4}. $  There exists $ \zeta >0$ such that for every  $ Q' \in B_{Q}^{3} (\zeta) \cap C_{c}(F) $, 
the Fr\'echet function $f_{Q'}$  has a unique minimum.
\end{theorem}

Here $B_Q^3(\zeta)$ denotes a ball of radius $\zeta$ around $Q$ in a Sobolev $3$-norm, and $C_c(F)$ denotes continuous functions with compact support in $F$. The proof is in Supplement C.

Secondly, we prove that any compact subset $E$ of $\ud$ admits an isometric embedding in a normed space 
$\mathcal B(E)$
such that any 
smooth distribution supported in $E$ has a unique Fr\'echet mean in $\mathcal B(E).$
 This Fr\'echet mean may not lie in the image of $E$.
 
\begin{theorem} \label{embedding} Let $Q$ be a continuous probability distribution with support in a compact subset $E$ of $\ud$ There is an isometric embedding $\iota$ of $E$ into $\mathcal B(E)$, the space of bounded functions on $E$ in the sup norm, such that the Fr\'echet function $f_Q$ extends to a function $\overline{f}_Q$ on the closed convex hull of 
$\iota(E)$ and such that $\overline{f}_Q$ has a unique minimum.
\end{theorem}
The proof is also in Supplement C.  
\\

We end this section with an example of the difficulties of handling a graph with a non-distinct weight vector.

\begin{example} \label{intolerable}
As in Figure 2,  take the distinct weight vector $\ve = (3,2,1).$  The fundamental domain is 
$F = F_{\ve}  = \{x\geq y\geq z\}.$  The proof of Corollary \ref{solidcor} shows that
every $\ve_1\in F_{a/2}$ has a neighborhood $U$ on which \begin{equation}\label{intol}
d_P([\ve_2],[\ve_3]) = d_E(\ve_2, \ve_3),\  {\rm for\ all}\ \ve_2, \ve_3\in U.
\end{equation}
In fact, we can take $U= F_{a/2}$ (See Example B.1 for a stronger statement.)  

We now show that the vector $(1,1,1)\in \partial F$ has no neighborhood $U$ in $F$ on which
(\ref{intol}) holds. Note that the stabilizer group of $(1,1,1)$ is all of $\Sigma_3.$

Take $\ve_2 = (1+a_1, 1+a_2, 1+a_3)$ with $a_1>a_2>a_3>0$, and $\ve_3 = 
 (1+b_1, 1+b_2, 1+b_3)$ with $b_1>b_2>b_3>0$.  For small $a$'s and $b$'s, these vectors will be arbitrarily close to $(1,1,1).$  Then
 $$d_E^2(\ve_2, \ve_3) = \sum_{i=1}^3 (a_i-b_i)^2.$$
 Take 
\begin{equation}\label{rest}
b_1\approx a_2, b_2\approx a_3\approx b_3
\end{equation}
 and let $\sigma = (132).$
 Then 
 \begin{equation}\label{est}d_P^2([\ve_2], [\ve_3]) < d_E^2(\ve_2, \ve_3)\end{equation}
 if 
 \begin{equation}\label{est2}\sum_{i=1}^3 (a_i-b_i)^2 > \sum_{i=1}^3 (a_i-b_{\sigma(i)})^2.
 \end{equation}
 The left hand side of (\ref{est2}) is approximately $(a_1-a_2)^2 + (a_2-a_3)^2$, and the right
 hand side is approximately $(a_1-a_3)^2.$  Thus (\ref{est2}) holds if
 $$(a_1-a_2)^2 + (a_2-a_3)^2 >(a_1-a_3)^2, i.e.\ {\rm if} \ a_2^2 + 2a_1a_3 > a_2 a_1 + a_2a_3.$$
 Since $a_2^2> a_2a_3$, we just need
 \begin{equation}\label{dec} 2a_3 > a_2.
 \end{equation}
 So once we choose the $a$'s and $b$'s to satisfy (\ref{rest}),(\ref{dec}), we obtain (\ref{est}), which proves the claim. 
 
We conclude  that there is no neighborhood $U$ of $(1,1,1)$ inside $\R^3$ (not just inside $F$) on which Euclidean distances in $U$ agree with Procrustean distances in $q(U)\subset \calU_3.$  This illustrates the impossibility of applying\\
 \cite{afsari11} to the singular point $[(1,1,1)]$ of $\calU_3$.
\end{example}


\section{Discussion}
\label{sec:disc}

Our work pertains to the geometric and statistical foundations of unlabeled networks. Specifically, we characterize the geometry of the space of such networks, define an appropriate notion of means or averages  of unlabeled networks, and derive the asymptotic behavior of the empirical average network.  This last result is a necessary precursor for the development of a variety of statistical inference tools in analogy to those encountered in a typical `Statistics 101' course, as we demonstrate in the context of hypothesis testing.  A key technical contribution of our work is that of providing broader conditions than available from general results on manifolds for  uniqueness of the Fr\'echet mean network of a distribution on the space of unlabeled networks. 

Our work here sets the stage for a program of additional research with multiple components.  Firstly, we expect that the asymptotic theory we develop can be extended in interesting ways. 
For example, it is unclear if the central limit theorem Theorem \ref{clt} is valid for either of our extensions (Theorems \ref{main theorem}, \ref{embedding}) to the main Theorem \ref{bestestimate} on uniqueness of the Fr\'echet mean.    For Theorem \ref{main theorem}, the obstacles to establishing a CLT  are technical differentiability issues, but there are deeper conceptual difficulties in the setting of Theorem \ref{embedding}.  
More broadly, when the uniqueness condition fails, there is a possibility of establishing a limit theorem based on the set distance between the sample Fr\'echet means and the population Fr\'echet means, in the spirit of analogous work done for the estimation of level sets~\cite{chen2016density,mason2009asymptotic}. However, the Hausdorff distance between the two does not necessarily go to zero,  as  in the counterexample in \cite{rabivic03}, so these limit theorems will be more subtle.  

Additionally,  while we have adopted Euclidean distance in our work here,  there are other norms (cut distance, edit distance, etc. \cite{lovasz2012large}) that merit consideration, and other Riemannian metrics may be more suitable for specific classes of  networks ({\it e.g.}, ego-networks).  Both directions will require more sophisticated techniques than the current paper: for Riemannian metrics, we expect to need Rauch-type comparison theorems as in \cite[Chapter 10]{docarmo}, while for norms like the cut distance that do not come from Riemannian metrics, we need very different metric geometry methods as in 
\cite{Burago}.   

In the particular context of ego-networks -- in the case where the ego node is included, and hence (uniquely among the nodes) labeled -- the existence of one labeled node reduces the symmetry group of the space of networks to the stabilizer subgroup of the node.  The Euclidean metric does not exploit this reduction of the symmetry group, whereas other Riemannian metrics more closely reflect this reduction.  For example, in the toy example where the ego-node is connected to only one other node, the hyperbolic metric of constant negative sectional curvature singles out the ego-node.  However, if more realistically the ego-node has degree greater than one, then most Riemannian metrics invariant under the stabilizer group will have both positive and negative sectional curvatures. It would be very interesting to see how our main uniqueness result (Theorem \ref{bestestimate}) generalizes to these metrics. 

It is also natural to expand our work to treat networks of different sizes.  The space $\ud = \calO_D/\Sigma_d$ of  unlabeled networks on $d$ nodes isometrically embeds into  $\mathcal U_{d+1}$ by adding to any labeled network in $\calO_d$ a new node with zero weights to all other nodes, and embedding $\Sigma_d$ into $\Sigma_{d+1}$ as the stabilizer subgroup of the new node.  We can then take the direct limit 
$\lim_{d\to\infty}\ud = \mathcal U_\infty$ as the space of all unlabeled networks. 
 It is unclear if a CLT can be expected for $\mathcal U_\infty$, but it should be possible to produce a CLT for the space of unlabeled networks of size at most a fixed constant.  
 
 There is also work to be done 
in calculating the Fr\'echet mean(s) for a sample of unlabeled networks and applying our results in practice (e.g., comparing the means of two subpopulations).  Calculation of the Fr\'echet mean is NP-hard, as a simple argument shows that a brute force approach requires $O((d!)^n)$ operations for $n$ samples of networks with $d$ nodes.  This is reinforced by the complexity of the algorithms discussed in the Appendices on calculations of fundamental domains and Fr\'echet integrals, due to the geometric complexity of the Procrustean distance on the space $\mathcal U_d$ of unlabeled networks.  Jain \cite{Jaina2016} provides an iterative algorithm for calculation of the Fr\'echet mean with complexity $O(n(d!))$, offering a major improvement.  Nevertheless, for large networks, this can still be expected to scale poorly.  Ultimately, however, in order to construct confidence regions and related hypothesis testing procedures, there is the computation of the covariance to consider, which, to the best of our knowledge, has yet to be addressed.  


Finally, although our work here assumed weighted undirected networks, it would be important to investigate how the finite set of undirected binary networks fits into our theory.  In particular, the definition and uniqueness of Fr\'echet means  needs to be elaborated. It is
 crucial to understand the placement of the  binary graphs inside the fundamental domains we constructed; preliminary work indicates that binary graphs are  unfortunately widely scattered throughout a fundamental domain.


\appendix

\section{Computing fundamental domains}\label{Jackson2}

In the two appendices, we discuss the computational difficulties of implementing the theory in Section~4.  In this appendix, we show 
that the fundamental domain $F_ {\vec w}$ is highly sensitive to the choice of distinct vector $\vec w$ as
axis.  

We use a Dirichlet fundamental domain for the action of $\Sigma_d$ on $\mathbb{R}_{\ge 0}^D$. By Lemma \ref{fd},
$F= F_{\vec w}= \bigcap_{\sigma \in \Sigma_{n}} \{z \in \mathbb{R}_{\ge 0}^{D} : d_{E}(\vec w,\vec z) \le d_{E}(\vec w,\sigma \vec z) \}$ for a distinct vector $\vec w \in \R_{\ge 0}^{D}$. This is the intersection of $d!+D-1 $ half-spaces, where the $D$ coordinate half-spaces are given by the inequalities $z_j \geq 0$. The cone
$F$   is a convex, non-compact polyhedral region in $\mathbb{R}_{\ge 0}^{D}$. The $d!-1 $ half-space regions are given by the linear inequalities:
$$\sum_{j=1}^{D} (w_{\sigma(j)}-w_{j})z_{j} \le 0,$$
for $\sigma \neq {\rm Id}.$

 Sage provides efficient tools for converting an input system of linear inequalities into a minimal description of the polyhedral output region; see Supplement D.
 
 We consider the simplest nontrivial example of graphs with four vertices: $ d=4$, $ D=\binom{d}{2}=6$.

\begin{example}   Choose the distinct vector $\vec  w=(1,2,3,4,5,6)$. Sage gives $F_{ w}$ as the intersection of 7 half-spaces or the convex hull of 7 rays [see Figure \ref{fig:sage_output_fund_domain}].

\begin{figure}[h]
\centering
\includegraphics[scale=.6]{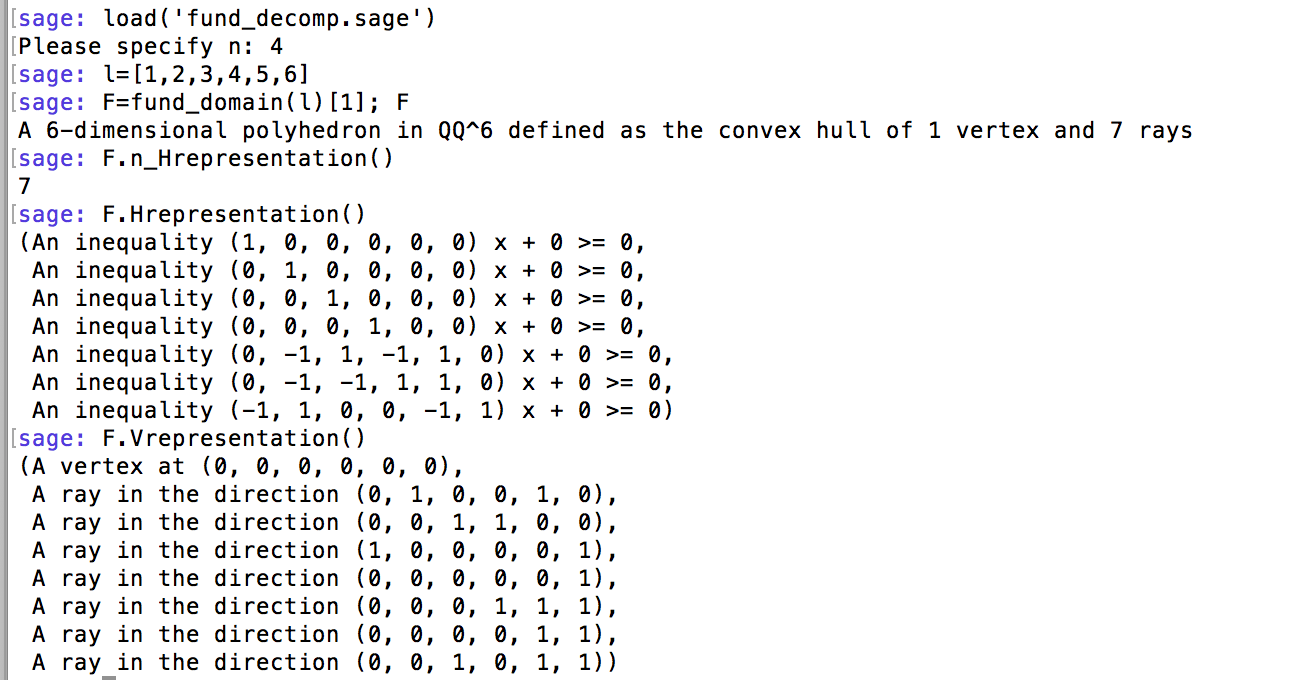}
\caption{Sage output for computation of a fundamental domain centered at the distinct vector $ \vec w=(1,2,3,4,5,6)$. $F.Hrepresentation()$ lists the half-spaces whose intersection is the fundamental domain. Note that we started with $4!+\binom{4}{2}-1 =29$ inequalities, and have narrowed it down to 7. $F.Vrepresentation()$ lists the 7 rays whose convex hull is the fundamental domain.}
\label{fig:sage_output_fund_domain}
\end{figure}
\end{example}

\begin{example} 
Now choose $\vec w=(1,2,3,4,5,6.1)$.  $F_{ \vec w}$ is now described as the convex hull of 79 rays, or the intersection of 18 half-spaces [see Figure \ref{fig:sage_output_fund_domain_2}].

\begin{figure}[h]
\centering
\includegraphics[scale=.67]{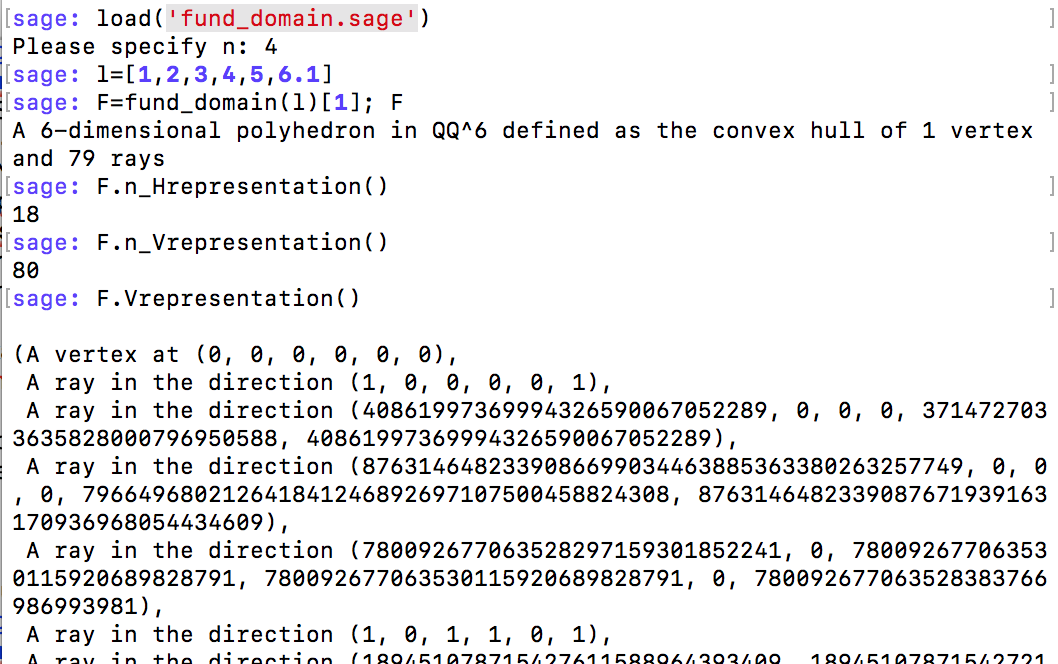}
\caption{A partial Sage output for computation of a fundamental domain centered at the distinct vector $\vec w=(1,2,3,4,5,6.1)$. $F.Hrepresentation()$ lists the half-spaces whose intersection is the fundamental domain. Note that we started with $4!+\binom{4}{2}-1 =29$ inequalities, and have narrowed it down to 18. $F.Vrepresentation()$ lists the 79 rays whose convex hull is the fundamental domain.}
\label{fig:sage_output_fund_domain_2}
\end{figure}
\end{example}

\section{Computation of the Fr\'echet integral}\label{Jackson3}

In this Appendix, we highlight the difficulty of computing the Fr\'echet integral 
$f([\vx]) =
\int_{\calU_d} d^2_P([\vx],[\vec y]) dQ'([\vec y]) $ for all $[\vec x]$. (Here $Q' = q_*Q$ in the notation of 
Thm~\ref{bestestimate}.)
Even when a
fundamental domain $F_{\vx}$ has been explicitly determined for a fixed $\vec x$, so that
 $f([\vec x]) =
 \int_{F_{\vx}} d^2_E(\vx,\vec y)dQ(\vec y),$ we have seen in Appendix \ref{Jackson2} that the shape of 
 $F_{\vx}$ depends delicately 
  on $[\vec x].$ 


We can instead divide $\mathbb{R}_{\ge 0}^{D}$ into $d!$ regions $F_{\sigma}(\vx) = 
\{\vec y\in \mathbb{R}_{\ge 0}^{D}: d_P(\left[\vx\right],\left[\vec y\right])=d_{E}(\vx,\sigma \cdot\vec
y)\}.$ Note that 
$F_{\sigma}(\vx)$ is  a fundamental domain for $\sigma \cdot \vx.$  Since
$\displaystyle F_{\vec w} = \bigcup_{\sigma \in \Sigma_{d}} F_{ \vec w} \cap F_{\sigma\cdot\vec x}$ for a fixed distinct vector $ \vec w$,
the Fr\'echet integral is given by
$$ f(\vec x) =\int_{{\mathcal U}_d} d_{P}^2(\left[ \vx\right],\left[\vec y \right]) dQ'(\left[\vec y\right])
=\sum_{\sigma \in \Sigma_{d}}\int_{F_{ \vec w} \cap F_{\sigma\cdot\vx}} d^2_{E}(\vx,\sigma\cdot
\vec y) dQ(\vec y). $$
Although on the right hand side we now have a  sum of Euclidean integrals for each $\vx$, as in Appendix \ref{Jackson2} 
it is difficult to explicitly compute $F_{ \vec w} \cap F_{\sigma\cdot \vx}$. 

We illustrate this computation in the simple case $d=3$.
Already for  $d=4$, the computation becomes too lengthy for inclusion here.

\begin{example} $d=3.$ First we show that the fundamental domain $F_{\vx}$ can be chosen to depend only on the ordering 
ord$(\vx)$
(e.g. from largest to smallest, as in Figure 2) of the components of a distinct vector $\vx$. The Procrustean distance between two points is $d_{P}\left(\left[\vx\right],\left[\vec y\right]\right)=\min_{\sigma \in \Sigma_{3}} d_{E}(\vx,\sigma\cdot\vec y)$. 
To minimize the Euclidean distance, we choose $\sigma$ which reorders $\vec y$ to match the ordering of $\vx$, as any other $\sigma$ cannot decrease the distance. (This uses the special fact that
 $\Sigma_{d=3}$ is the full permutation group of the $D=3$ set of weight vectors). Therefore,  we can choose 
  $F = F_{\vx}=\{ \vec y \in \mathbb{R}_{\ge 0}^{3} : d_{E}(\vx,\vec y) = d_{P}([\vx],[\vec y]) \}=\{\vec y \in \mathbb{R}_{\ge 0}^{3} : \text{ord}(\vec y)=\text{ord}(\vx)\}$ to be
  independent of the distinct vector $\vec x$. 
 The 
  Fr\'{e}chet integral 
  for a compactly supported probability measure $Q'$ on $\calU_3$ (or equivalently for $Q$ on $F$)
  equals
\begin{align*}f(\left[\vx\right])&=\int_{F} d_{E}(\vec x,\vec y)^{2} dQ(\vec y)\\
&=\Vert \vx\Vert^{2}\int_{F} dQ(\vec y) -2\vec x \cdot \int_{F} \vec y \ dQ(\vec y) + \int_{F} \Vert \vec y
\Vert^{2} dQ(\vec y)\\
&= \Vert \vx\Vert^2 -2 \vx\cdot \int_{F} \vec y\ dQ(\vec y)  + B\\
& =  \Vert \vx\Vert^2 -2 \sum_i C_i x^i  + B,\\
\end{align*}
where the second integral on the last line is the dot product of a vector and a vector valued integral,  $C_i = \int_{F} y^i dQ(\vec y)$, and $B = \int_{F} \Vert \vec y
\Vert^{2} dQ(\vec y).$
Thus $f$ is quadratic in $\vx$ on distinct vectors.  Since the distinct vectors are dense in Euclidean space, $f$ is quadratic on all vectors. 
Therefore, $f$ is strictly convex. Since we are minimizing over a convex region, $f$ has a unique global minimum. To explicitly compute it, note that $F$ has eight strata in varying dimensions 0, 1, 2, and 3; these are labeled $(1),\ldots,
 (8)$ in the table below. The restriction to each stratum is smooth away from the lower dimensional boundaries, so we can simply minimize on each open piece to find eight local minima $x^{*}$.
\medskip

\begin{center}
\begin{tabular}{ |c|c|c|c| } 
 \hline
 \# & \text{dim} & \text{Region} & $x^{*}=$ \\
 \hline\hline
  (1) & 3 & $x_{3} \ge 0$ and $0 < x_{1} < x_{2} < x_{3}$ & $
  \left(C_{1},C_{2},C_{3}\right)$ \\
  \hline
  (2) & 2 & $x_{1}=0$ and $0 < x_{2} < x_{3}$ & $
  \left(0,C_{2},C_{3}\right)$  \\
  \hline
  (3) & 2 & $0 < x_{1}=x_{2} < x_{3}$ & $
 \left(\frac{1}{2}(C_{1}+C_{2}),\frac{1}{2}(C_{1}+C_{2}),C_{3}\right)$\\
  \hline
  (4) & 2 & $0 < x_{1} < x_{2}=x_{3}$ & $
  \left(C_{1},\frac{1}{2}(C_{2}+C_{3}),\frac{1}{2}(C_{2}+C_{3})\right)$\\
  \hline
  (5) & 1 & $x_{1}=x_{2}=0$ and $x_{3} > 0$ & $
  \left(0,0,C_{3}\right)$ \\
  \hline
  (6) & 1 & $x_{1}=0$ and $0 < x_{2}=x_{3}$ & $
  \left(0,\frac{1}{2}(C_{2}+C_{3}),\frac{1}{2}(C_{2}+C_{3})\right)$ \\
  \hline
  (7) & 1 & $0 < x_{1}=x_{2}=x_{3}$ & $
  \left(C', C', C'\right)$\\
  \hline
  (8) & 0 & $x_{1}=x_{2}=x_{3}=0$ & (0,0,0) \\
 \hline
\end{tabular}
\end{center}
\medskip

Here $C' = (1/3) (C_{1}+C_{2}+C_{3}).$
The true global minimum will depend on the values of $C_{1}$, $C_{2}$, and $C_{3}$.

\end{example}

\section*{Acknowledgements}

We would like to thank the referees for very helpful suggestions. This research is funded in part by an ARO grant W911NF-15-1-0440. In addition, LL was supported by NSF grants  IIS 1663870 and DMS Career 1654579, and DARPA grant N66001-17-1-4041; EK was supported by AFOSR grant 12RSL042.

\centerline{SUPPLEMENTARY MATERIAL}
All the supplementary materials (see a brief descritpion below) can be found in
{\tt https://github.com/KolaczykResearch/Ave-Unlab-Nets}.

{\bf Supplement A:}
A description of $\calU_3$, a sketch of Theorem \ref{stratspace}, and results on the topology of
$\ud$

{\bf Supplement B: \label{suppB}}
Proof of Proposition \ref{4.6}.

{\bf Supplement C:}
Improvements to the main Theorem \ref{bestestimate}

{\bf Supplement D:}
Sage code for computing fundamental domains

{\bf Supplement E:}
Color figures

\bibliographystyle{imsart-nameyear}
\bibliography{reference}

\end{document}